\documentclass[12pt]{amsart}

\title{Tensor stable moduli stacks and refined representations of quivers}

\author{Tarig Abdelgadir}
\address{The Abdus Salam International Centre for Theoretical Physics, 
Strada Costiera 11, 
Trieste 34151, 
Italy
}
\email{tarig.m.h.abdelgadir@gmail.com}

\author{Daniel Chan}
\address{School of Mathematics and Statistics, 
UNSW Sydney, 
NSW 2052,	
Australia
}
\email{danielc@unsw.edu.au}

\usepackage{mymacros,amssymb}
\usepackage{fullpage,tikz}
\usepackage{tikz-cd}
\usepackage{pst-node}
\usetikzlibrary{graphs,angles,quotes,positioning,calc,arrows.meta}
\usepackage[all]{xy}

\begin{document}
\maketitle

\begin{abstract}
In this paper, we look at the problem of modular realisations of derived equivalences, and more generally, the problem of recovering a Deligne-Mumford stack $\bX$ and a bundle $\cT$ on it, via some moduli problem (on $\bX$ or $A = \End_{\bX} \cT$). The key issue is, how does one incorporate some of the monoidal structure of $\Coh(\bX)$ into the moduli problem. To this end, we introduce a new moduli stack, the tensor stable moduli stack which generalises the notion of the Serre-stable moduli stack. We then show how it can be used both for stack recovery and the modular realisation problem for derived equivalences. We also study the moduli of refined representations and how it addresses these problems. Finally, we relate the two approaches when $\cT$ is a tilting bundle which is a direct sum of line bundles. 
\end{abstract}

\section{Introduction}

Moduli spaces are a fruitful way to study a $k$-linear abelian category {\sf C}. The McKay correspondence provides one spectacular example. Here, we let $G \subset \SL(2,k)$ be a finite group and take {\sf C} to be the category of $G$-equivariant $k[x,y]$-modules. There is a natural moduli space of objects in {\sf C}, corresponding to some natural discrete invariant, which recovers the minimal resolution of $\bA^2/G$. Another highly influential example is the point scheme of Artin-Tate-van den Bergh \cite{ATV90}. Here {\sf C} is the category of graded modules over some noncommutative graded algebra. The point scheme is the rigidified moduli space of ``point modules'' in {\sf C}. In the case of the 3-dimensional Sklyanin algebra, this point scheme is an elliptic curve, and the algebra can be fully analysed since it has a codimension one quotient which is a twisted version of a homogeneous co-ordinate ring on the elliptic curve.

This raises the natural question, what moduli problems should one study, and the above examples may tempt one to think that the only interesting one is the (rigidified) moduli of isomorphism classes of objects in {\sf C}. However, this is far from the case, and the easiest way to appreciate this is the {\em tautological moduli problem}: How do you recover a Deligne-Mumford (DM) stack $\bX$ as a (rigidified) moduli problem on $\text{\sf C}=\Coh(\bX)$? As stated in this generality, this problem is doomed to failure, for objects in abelian categories have connected automorphism groups whereas the stabiliser groups of DM stacks are finite groups. A deeper reason is that Gabriel-Rosenberg's theorem for reconstructing schemes from their abelian categories of quasi-coherent sheaves fails for DM stacks.

Lurie's Tannakian duality for geometric stacks \cite{Lurie} provides the inspiration to get around these issues: \cite[Theorem 5.11]{Lurie} states that one may recover $\bX$ from its monoidal abelian category $(\Coh(\bX), \otimes)$. Therefore the key to solving the tautological moduli problem is to remember some of the monoidal structure of  $\Coh(\bX)$. One of the main goals of this paper is to construct and study interesting moduli stacks on an abelian category {\sf C} which somehow incorporate residues of this monoidal structure in {\sf C} (note {\sf C} is not necessarily monoidal). Gabriel-Rosenberg's theorem shows that we need not remember all the monoidal structure, since we can safely ignore all monoidal structure in the case of schemes. The moduli stacks of interest for us are those which address either the tautological moduli problem above, or the related problem of {\em modular realisations of derived equivalences}. The latter seeks to express a derived equivalence between {\sf C} and $\bX$ as a Fourier-Mukai transform whose kernel is the universal bundle for some moduli problem on {\sf C}. 


The first moduli stack we construct  generalises the rather ad hoc Serre stable moduli stack introduced in  Chan-Lerner's \cite{CL}. To motivate this, we consider the case where $\bX=X$ is a separated scheme and  there is a simple solution to the tautological moduli problem. Take $\Delta \subset X \times X$ to be the diagonal, then $\cO_{\Delta}$ is the universal skyscraper sheaf on $X$. For stacks, it is not clear what the correct notion of a skyscraper sheaf is since they can ``fractionate'' in the language of physics. In Section~\ref{sec:sky}, we address this problem and define the moduli stack $\bM$ of skyscraper sheaves (suitably rigidified to remove the common $k^{\times}$ in automorphism groups). For example, the stack $B \mu_2 := [\Spec k / \mu_2]$ has two line bundles $k_0$ and $k_-$ corresponding to the trivial and non-trivial characters of $\mu_2$. For us, a skyscraper sheaf on $B\mu_2$ will be the direct sum $k_0 \oplus k_-$.

To incorporate the monoidal structure, we consider line bundles $\cL_1,\ldots, \cL_s \in \Pic \bX$ and the corresponding rationally defined self-maps $\cL_i \otimes_{\bX} (?) \colon \bM \dashrightarrow \bM$. The {\em tensor stable moduli stack} $\bM^{\cL_1,\ldots,\cL_s}$ is defined in Definition~\ref{defn:doubletensor} as a mild modification of the simultaneous fixed point stack. When $\bX$ is a scheme, the maps $\cL_i \otimes_{\bX} (?)$ are the identity so the procedure is trivial, but this is no longer the case when $\bX$ has stacky points. One of our main results is the following rather imprecise re-statement of Theorem~\ref{thm:tautological} which vastly generalises the rather ad hoc \cite[Theorem~9.6]{CL}.
\begin{theorem}
Let $\bX$ be a separated quasi-projective stack and suppose that $\cL_1\oplus \ldots \oplus \cL_s$ is faithful. Then $\bM^{\cL_1,\ldots,\cL_s} \simeq \bX$.
\end{theorem}
\noindent 
The simplest example of this theorem in the stacky case is where $\bX = B\mu_2$ which can be recovered as $\bM^{k_-}$.  Note that the tensor stable moduli stack here parametrises a single object consisting of an isomorphism $k_0 \oplus k_- \simeq k_- \otimes (k_0 \oplus k_-)$. Taking into account of rigidification, this boils down to an isomorphism $k_0 \simeq k_- \otimes k_-$ which has automorphism group $\mu_2$ as we show in Example~\ref{eg:BmuRecovery}. 
Returning to the general case in the theorem, the hypothesis on $\cL_1\oplus \ldots \oplus\cL_s$ implies that the stabiliser groups of $\bX$ are all abelian. Fortunately, this is good enough for applications to the weighted (a.k.a. GL) projective spaces of Herschend-Iyama-Minamoto-Opperman \cite{HIMO}, an example we look at in Section~\ref{sec:HIMO}. It also applies to the weighted projective surfaces studied in \cite{C17}. We remark that, incorporating monoidal structure in this way, also solves the aforementioned issue of connected automorphism groups in a manner reminiscent of classical Tannakian duality for finite groups.

We turn our attention now to the modular realisation problem, which was recently solved in the special case of Geigle-Lenzing \cite{GL} derived equivalences in \cite{CL} and Abdelgadir-Ueda's \cite{AU}. To this end, let $\bX$ be a DM stack that possesses a tilting bundle $\cT$. Consider the endomorphism algebra $A = \End_{\bX} \cT$ which can be written as $kQ/I$ for some quiver $Q$ and admissible ideal $I$. Our starting point is the moduli stack $\bM$ of $A$-modules of dimension vector $\vec{d}$, rigidified so as to remove the common automorphism group $\bG_m$. 

One natural approach we follow is to use the derived equivalence $F = \RHom_{\bX}(\cT,?)$ to try to transfer the tautological moduli problem to one realising $F$. This works for the tensor stable moduli stack as follows. The functors $\cL_i \otimes_{\bX}(?)$ now correspond to two-sided partial tilting complexes $L_i$ over $A$, and the only monoidal structure on $\Coh(\bX)$ we remember are the auto-functors $(?) \otimes^L_A L_i$ on $D^b(A)$. These now give partially defined self-maps on $\bM$ and $\bM^{L_1,\ldots,L_s}$ can be defined as before. This leads to Theorem~\ref{thm:stackfromquiver} which we state imprecisely as,
\begin{theorem}  \label{thm:2}
Under ampleness conditions on $\cL_i, \cT$, there is a natural isomorphism  $\bX \simeq \bM^{L_1,\ldots,L_s}$. Furthermore, the data of the universal object includes the universal $A$-module $\cT^{\vee}$, which can be used to realise the derived equivalence as a Fourier-Mukai transform. 
\end{theorem}
The higher dimensional generalisation \cite{HIMO} of the Geigle-Lenzing derived equivalence can in particular, be realised using this theorem. 
We remark that, in the case where $\bX$ is a scheme, the result here is similar to Bergmann-Proudfoot's \cite{MR2421120}. It is, however, important to note that our theorem and approach use the full force and naturality of moduli theory, and do not rely on the restrictive smoothness or pointwise arguments found in \cite{MR2421120}.

We next look at the approach of {\em refined representations} as per \cite{AU}, which incorporates the monoidal structure in a rather different way. It is convenient now to restrict to the case where $\cT$ is a direct sum of line bundles $\cT_i$, but not necessarily a tilting bundle, and the dimension vector is the constant $\vec{1}$. These starting hypotheses and the general approach here have their origin in Craw-Smith's \cite{Craw-Smith} theory of multilinear series and Abdelgadir's generalisation to toric stacks \cite{Abd}. Whereas they were interested in embedding stacks in moduli spaces, we seek to recover $\bX$ as a moduli problem on ${\sf C} = A-\text{mod}$. Of course, there needs to be conditions on $\cT$, and tilting is but one we consider.
In the refined representations approach, families of $A$-modules are enhanced with refinement data $g$ (see Section~\ref{sec:refined}). 
This data is designed precisely to ensure that relations among the partial tilting bundles $\cT_i$ in $\Pic(\bX)$ are preserved by the universal bundles $\cU_i$.
Thus, the monoidal data structure we remember here includes the kernel $\Lambda_r$ of the composite $\bZ^{Q_0} \to \Pic \bM \to \Pic \bX$ which sends the $i$-th universal bundle $\cU_i$ to $\cT_i$.

More precisely, we have (ignoring rigidification issues here for clarity), that the universal bundles induce a morphism $\bM \to B(\bG_m^{Q_0})$. By restricting to $\Lambda_r \subset \bZ^{Q_0}$ we get a morphism $\bM \to B\Lambda_r^{\vee}$ where $\Lambda_r^{\vee}$ is the dual group of $\Lambda_r$. 
Taking base change with $\textup{pt} \to B\Lambda_r^{\vee}$, results in a fibre product stack $\widetilde{\bM}$ which, by formal nonsense, parametrises $A$-modules $\cM$ enhanced by the required refinement data $g$.
This enhancement of $A$-modules has the desired effect on the universal bundles and reduces the stabiliser groups of refined representations to subgroups of the dual group $(\Pic \bX)^\vee$ as opposed to the larger $\bG_m^{Q_0}$. Note that if $i,i',j,j' \in Q_0$ are such that the Peirce components $\Hom_{\bX}(\cT_i,\cT_j), \Hom_{\bX}(\cT_{i'},\cT_{j'})$ of $A$ are isomorphic we get a natural element in $\Lambda_r$, and as part of the setup, we will need to pick compatible isomorphisms between these Peirce components. The moduli stack $\bM_{\textup{ref}}$ of refined representations is the closed substack of $\widetilde{\bM}$ consisting of pairs $(\cM,g)$ compatible with these isomorphisms (see Definition~\ref{def:refined}).

We now paraphrase Corollary~\ref{cor:stackfromCox}. We will need to assume that $\bX$ is a Mori dream stack so has a finitely generated Cox ring $R$. The key hypothesis is that $\cT$ {\em captures the Cox ring} in the sense of Definition~\ref{def:captureCox}. Loosely speaking, this condition can be understood as follows. The Peirce components of $A$ are of the form $\Hom(\cT_i,\cT_j) \simeq H^0(\cT_j \otimes \cT_i^{-1})$ which is one of the graded components of $R$. Similarly, some of the multiplication of $R$ is reflected in the multiplication in $A$. If $A$ contains enough information to recover the generators and relations in $R$, then we say $\cT$ captures the Cox ring. 
\begin{theorem}  \label{thm:3}
Suppose $\bX$ is a Mori dream stack and that $\cT$ captures the Cox ring of $\bX$. 
Then there is an open subset $\bM_{\textup{ref}}^\circ$ of $\bM_{\textup{ref}}$ defined by a GIT stability parameter which is isomorphic to $\bX$.
\end{theorem}
The theorem applies in the case of GL projective spaces and encapsulates when the ad hoc argument in \cite{AU} works. Checking the hypothesis that $\cT$ captures the Cox ring is usually easy for concrete examples, but not necessarily in the abstract. What is striking is how different this approach to stack recovery is to the one using tensor stable moduli stacks. Even the hypotheses for when they work are very different. This raises a host of interesting questions like, is there a relationship between the way the two approaches incorporate the monoidal structure?

We provide an answer when $\cT$ is both a tilting bundle and a direct sum of line bundles. We pick the $L_i$ to correspond to the $\cT_i$. We remark first that the tensor stable moduli stack parametrises $A$-modules $\cM$ enhanced with {\em tensor stability data} which we will denote by $\psi$.
\begin{theorem}  \label{thm:samemoduli}
With the above hypotheses, there is an isomorphism of $\bM^{L_1,\ldots,L_s}$ with an open subset of $\bM_{\textup{ref}}$ and moreover, the inverse isomorphisms can be defined explicitly by a procedure which tells you how to go back and forth between tensor stability data $\psi$ and refinement data $g$.
\end{theorem}
This relates the approaches developed in \cite{AU} and \cite{CL}. There are many reasons why this theorem is interesting. Firstly, it is not clear if the hypotheses of Theorem~\ref{thm:2} imply those of Theorem~\ref{thm:3}, so we now know more instances where refined representations give modular realisations of derived equivalences. Secondly, one advantage the moduli of refined representations has over the tensor stable moduli stack, is that the former is naturally a quotient stack and can be studied via GIT whereas the latter is not obviously so. The relationship between tensor stability and refinement data may suggest a way to apply GIT to tensor stable moduli stacks. Finally, we do expect the open subset $\bM_\text{ref}^\circ$ referred to in Theorem~\ref{thm:samemoduli} to be defined by a GIT stability parameter $\theta$.

We end with some concluding remarks about the modular realisation of derived equivalences problem. The representation theory of a finite dimensional algebra $A$ is extremely subtle as, for example, its complexity varies chaotically as you vary generators and relations for $A$. The beauty of derived equivalences is that they now can be used to explain why some algebras have nice representation theory. Indeed, Geigle-Lenzing formulated their derived equivalence to explain Ringel's study of the representation theory of canonical algebras. The natural question is thus, given $A$, how might you pick a DM stack $\bX$ which is derived equivalent to it. The old approach is to guess $\bX$ (perhaps using knowledge of $K_0(A)$) and a tilting bundle on it. Solutions to the modular realisation problem provide a much more elegant approach, since they are essentially machines for churning out candidates for both $\bX$ and the tilting bundle. There is still some guesswork in the choice of $\vec{d}, L_i$ or $\Lambda_r$ etc. Regardless, the machine will always produce a universal $A$-module, and hence adjoint functors relating $A$ to a moduli stack. Even if these functors are not derived equivalences, they may retain enough information to be useful, as happens in the case of point schemes for Sklyanin algebras.  

\subsubsection*{Notation and conventions}
We work over an algebraically closed field of characteristic 0.
For a finite set $S$ we will use $\bZ^S$ to denote the free abelian group generated by $S$ and $\chi_s \in \bZ^S$ to denote the generator corresponding to $s \in S$.

\subsection*{Acknowledgements}
We would like to thank Shinnosuke Okawa and \v{S}pela \v{S}penko for helpful discussions.
We are also grateful to Michael Wemyss for his question about the relation of this work to that of Bergmann-Proudfoot \cite{MR2421120}.
This work was done under the funding of the Australian Research Council's Discovery Projects grant number DP130100513.

\section{Background: quasi-projective stacks}

In this paper, we will mainly be dealing with quasi-projective stacks $\bX$ as defined by Kresch \cite{Kr}. We record here his definition as well as some basic facts about such stacks.

Let $\bX$ be a Deligne-Mumford stack of finite type. We will also assume that $\bX$ is separated, or more generally, has {\em finite inertia}, which just means that the inertia stack $$\mathbb{I}(\bX) :=\bX \times_{\Delta,\bX \times \bX , \Delta} \bX$$ 
is finite over $\bX$. We know from \cite{KeM} that there is a coarse moduli space $X$ and we let $c\colon \bX \to X$ denote that canonical quotient morphism. Furthermore, \'etale locally on $X$, $\bX$ is isomorphic to a quotient stack of the form $[U/G]$ where $G$ is a finite group.

Our characteristic 0 assumption ensures that $\bX$ is a tame stack in the sense of \cite{AOV} so in particular, in the terminology of \cite[Definition~3.1]{Alp}, the morphism $c \colon \bX \to X$ is {\em cohomologically affine} in the sense that $c$ is quasi-compact and $c_* \colon \Qcoh(\bX) \to \Qcoh (X)$ is exact.  From \cite[Proposition~3.10]{Alp}, we know that cohomologically affine morphisms are stable under compositions and base change if the bases are Deligne-Mumford stacks. 

We recall Alper's projection formula \cite[Proposition~4.5]{Alp}.
\begin{proposition}   \label{prop:projection}
Let $f \colon \bY \to Y$ be a cohomologically affine morphism of Artin stacks where $Y$ is an algebraic space. Then the natural morphism below is an isomorphism for quasi-coherent sheaves $\cF, \cF'$ on $\bY, Y$, respectively.
$$ f_* \cF \otimes_{Y} \cF' \to f_* (\cF \otimes_{\bY} f^* \cF').$$
\end{proposition}

Following \cite{OS} we define
\begin{definition}
A coherent locally free sheaf $\cG$ on $\bX$ is a {\em generating sheaf} if the natural morphism
$$ c^*(c_*\cHom_{\bX}(\cG, \cF)) \otimes_{\bX} \cF \to \cF$$
is surjective for every quasi-coherent sheaf $\cF$ on $\bX$. 
\end{definition}
By \cite[Theorem~5.2]{OS}, this condition can be checked geometrically pointwise as follows. Given any geometric point $\xi \colon \operatorname{Spec} k \to \bX$, we consider the (geometric) stabiliser group $G_{\xi} := \operatorname{Spec} k \times_{\xi, \bX} \mathbb{I}(\bX)$ which is a finite group. Then a locally free sheaf $\cG$ generates if and only if the $G_{\xi}$-module $\xi^* \cG$ generates $\textup{Mod-} kG_{\xi}$ for every geometric point $\xi$.

The importance of this concept for us, is that it allows us to relate the theory of stacks to non-commutative algebraic geometry. 
\begin{proposition} \label{prop:Morita}
Fix a generating sheaf $\cG \in \Qcoh(\bX)$ and define $\cA := c_* \cEnd (\cG)$ then 
\begin{align*}
\Phi \colon \Qcoh (\bX) &\longrightarrow \textup{Mod-}\cA \\
 \cF & \longmapsto c_* \cHom(\cG, \cF)
\end{align*}
is an equivalence of categories.
\end{proposition}

\begin{proof}
First note that $\Phi$ is exact and that $\cG$ is locally free. Moreover, $\Phi$ admits a left adjoint:
\begin{align*}
\Psi \colon \textup{Mod-}\cA &\longrightarrow \Qcoh (\bX) \\
 M & \longmapsto c^*M \otimes_{c^* \cA} \, \cG.
\end{align*}
We begin by showing that the composite $\Phi \circ \Psi$ is isomorphic to the identity. 
Since $\cG$ is a generating sheaf, we may present a general $\cF \in \Qcoh(\bX)$ as follows: $$c^* \cV_1 \otimes_\bX \cG \longrightarrow c^* \cV_2 \otimes_\bX \cG \longrightarrow \cF$$ where $\cV_1, \cV_2 \in \Qcoh(X)$ are locally free. 
Hence it suffices to show that the adjuction morphism $\Phi \circ \Psi (c^* \cV \otimes_\bX \cG) \longrightarrow c^* \cV \otimes_\bX \cG$ is in fact an isomorphism.
This then follows from the following chain of isomorphisms:
\begin{align*}
\Phi \circ \Psi (c^* \cV \otimes_\bX \cG) &\cong c^*(c_* \cHom(\cG, c^* \cV \otimes_\bX \cG) \otimes_{c^* \cA} \, \cG \\
& \cong c^*(c_* (c^* \cV \otimes_\bX \cEnd(\cG))) \otimes_{c^* \cA} \, \cG \\
& \cong c^* (\cV \otimes_X \cA) \otimes_{c^* \cA} \, \cG \\
& \cong c^*\cV \otimes_\bX \cG.
\end{align*}

It remains to show that $\Psi \circ \Phi$ is isomorphic to the identity.
For $M \in \textup{Mod-}\cA$, one may show that the adjunction morphism $M \rightarrow \Psi \circ \Phi(M)$ is an isomorphism it suffices by checking it locally on the coarse moduli space. 
Locally over the coarse moduli space we may freely present $M \in \textup{Mod-}\cA$ as follows: $$\cA^m \longrightarrow \cA^n \longrightarrow M \longrightarrow 0.$$ 
The result then follows from the observation that $A \longrightarrow \Psi \circ \Phi (\cA)$ is an isomorphism.
\end{proof}

\begin{remark}  \label{rem:Morita}
Many results about quasi-coherent sheaves on stacks can be reduced corresponding results on schemes using Proposition~\ref{prop:Morita}. This seems to be a relatively easy way to check results concerning stacks for those less familiar with the theory. We will use this a number of times.
\end{remark}

\begin{definition}
We say $\bX$ is {\em (quasi-)projective} if it has a generating sheaf and the coarse moduli space $X$ is a (quasi-)projective scheme.
\end{definition}

Examples of projective stacks include weighted projective lines as introduced by Geigle-Lenzing \cite{GL} and the class of stacks introduced in Herschend-Iyama-Minamoto-Opperman \cite{HIMO}.
The authors of \cite{HIMO} chose to name them Geigle-Lenzing projective spaces in honour of \cite{GL}.

\section{Skyscraper sheaves on stacks}
\label{sec:sky}

In this section, we introduce the notion of a skyscraper sheaves on a quasi-projective stack $\bX$ and their moduli. By definition, the coarse moduli space $X$ is a quasi-projective scheme and there exists a generating sheaf $\cG$ which is far from being unique. Let $c \colon \bX \rightarrow X$ be the canonical morphism to the coarse moduli scheme and suppose there is a fixed decomposition $\cG = \oplus_i \cG_i$. Let $\cA := c_* \cEnd_{\bX} \cG$ which is a finite sheaf of algebras on $X$.

Let $\widetilde{\bM}_{\Coh}$ be the moduli stack of coherent sheaves on $\bX$. Recall that $\Coh(\bX)$ is a $k$-linear category, so the inertia groups of every object contain a copy of $\mathbb{G}_m$. We remove this common copy of $\mathbb{G}_m$ by {\em rigidification} as defined for example in \cite[Section~5]{ACV} (see also \cite[Section~2.3]{CL} for a gentle description). Let $\bM_{\Coh}$ denote the resulting rigidified moduli stack of coherent sheaves on $\bX$. It can be described as the stackification of the following pre-stack $\bM^{pre}$. Given a test scheme $T$, the objects of $\bM^{pre}(T)$ are those of $\tilde{\bM}(T)$. Given objects $\cM, \cN \in \bM^{pre}(T)$, the isomorphisms from $\cM$ to $\cN$ consist of equivalence classes of isomorphisms $\phi \colon \cM \to \cM' \otimes_T \cN$ where $\cN$ is a line bundle on $T$ and $\phi' \colon \cM \to \cM' \otimes_T \cN'$ is {\em equivalent} to $\phi$ if there is some isomorphism $l \colon \cN \to \cN'$ with $\phi' = (\id \otimes l) \phi$. The pullback pseudo-functor is that induced from $\widetilde{\bM}_{\Coh}$. 

We will always rigidify our moduli stacks in this way so in general will omit the adjective ``rigidified''. We next define the (rigidified) moduli stack $\bM_{\textup{Fin}}$ of finite length sheaves on $\bX$. The objects over a test scheme $T$ consist of coherent sheaves $\cF \in \Qcoh(\bX \times T)$ which are flat over $T$ and such that the support $Z \subseteq X \times T$ of $\Phi_{\cG}(\cF):=(c\times \id_T)_*\cHom_{\bX}(\cG, \cF)$ is finite over $T$. Since flatness is a property of a Grothendieck category (see \cite{AZ01}), this is equivalent by Proposition~\ref{prop:Morita}, to a flat family of coherent $\cA$-modules with finite support. The morphisms in $\bM_{\textup{Fin}}$ are defined to be the same as for the moduli stack of coherent sheaves on $\bX$. We note that the definition is independent of the choice of generator, since Morita equivalences of finite sheaves of algebras over $X$ preserve support on $X$.

\begin{proposition}  \label{prop:MFinisartin}
The moduli stack $\bM_\textup{Fin}$ is an Artin stack.
\end{proposition}
\begin{proof}
Suppose first that $\bX$ and hence $X$ are projective so that we may speak of Hilbert polynomials of coherent $\cA$-modules with respect to some fixed choice of an ample line bundle. We may think of $\bM_{\Coh}$ as the moduli stack of coherent $\cA$-modules and then $\bM_{\textup{Fin}}$ consists of those components of $\bM_{\Coh}$ whose Hilbert polynomials are bounded. 

For quasi-projective $X$, we may pick a projective closure $\bar{X}$ of $X$ and extend $\cA$ to a coherent sheaf of algebras $\bar{\cA}$ on $\bar{X}$. It suffices then to show that $\bM := \bM_{\textup{Fin}}$ is an open substack of the moduli stack $\bar{\bM}$ of finite length $\bar{\cA}$-modules. 
Let $\cF$ be a flat family of $\cA$-modules over $T$ whose support $Z \subseteq X \times T$ is finite $i\colon X \times T \to \bar{X} \times T$ and $j \colon Z \to \bar{X} \times T$ be the natural embeddings. Then $i_* \cF = j_* \cF$ defines a flat family of $\bar{\cA}$-modules of finite length. This exhibits $\bM$ as a substack of $\bar{\bM}$ and it only remains to observe that the condition of being in $\bM$ is open. 
\end{proof}

To obtain analogues of the notion of skyscraper sheaves, we need some discrete invariants. The following helps us define such invariants. 
\begin{proposition}  \label{prop:cstarflat}
Let $\cF \in \Coh(\bX \times T)$ be a family of sheaves which is flat over $T$ and $\pi_1 \colon \bX \times T \to \bX, \pi_2 \colon \bX \times T \to T$ be the projection maps. Then for any coherent locally free sheaf $\cV \in \Coh(\bX)$, the sheaf $(c\times \id_T)_*\cHom_{\bX \times T}(\pi_1^*\cV, \cF)$ is flat over $T$. In particular, if $\cF$ is a flat family of finite length sheaves, then $\pi_{2*}\cHom_{\bX \times T}(\pi_1^*\cV, \cF)$ is a locally free sheaf on $T$. 
\end{proposition}
\begin{proof}
Since $\cHom_{\bX \times T}(\pi^*\cV, \cF)$ is also flat over $T$, it suffices, for the first assertion, to show that $(c \times \id_T)_* \cF$ is flat over $T$. To this end, consider an injection of quasi-coherent sheaves $M' \hookrightarrow M$ on $T$. Flatness of $\cF$ means that we have an injection $\cF \otimes_{\bX \times T} \pi_2^* M' \hookrightarrow \cF \otimes_{\bX \times T} \pi_2^* M$. The projection formula in Proposition~\ref{prop:projection} and the fact that $c \times \id_T$ is cohomologically affine now shows that the natural map
$$ (c \times \id_T)_* \cF \otimes_{X \times T} \pi_2^*M' \hookrightarrow (c \times \id_T)_* \cF \otimes_{X \times T} \pi_2^*M$$
is injective. Hence $(c \times \id_T)_* \cF$ is indeed flat over $T$. The second assertion now follows from \cite[Proposition~9.2(d)]{Hart}. 
\end{proof}

The proposition allows us to make the following definition. 
\begin{definition}
Let $\cF\in \cM_{\textup{Fin}}(T)$ be a flat family of finite length sheaves on $\bX$ over $T$ and $\cV$ be a coherent locally free sheaf on $\bX$. The {\em $\cV$-rank of $\cF$} is defined to be 
$$\cV\!-\!\rank \cF := \rank_T \pi_{2*} \cHom_{\bX \times T}(\pi_1^* \cV, \cF).$$
We say that $\cF$ has {\em skyscraper $\cV$-rank} if $\cV\!-\!\rank \cF = \rank \cV$. We define the {\em moduli stack $\bM_{\textup{Sky}}$ (resp. $\bM_{\cG-\textup{Sky}}$) of skyscraper sheaves (resp. relative to $\cG = \oplus \cG_i$)}, to be the substack of $\bM_{\textup{Fin}}$ consisting of $\cF \in \bM_{\textup{Fin}}$ with skyscraper $\cV$-rank for every coherent locally free sheaf $\cV$ on $\bX$ (resp. for all $\cV = \cG_i$). 
\end{definition}

The $\cV$-rank is an important discrete invariant we can use to decompose the moduli stack of finite length sheaves. These invariants are of course, not all independent and it is useful to know the relationships between them. For example, we have

\begin{proposition}  \label{prop:Vrankcoarse}
Let $\cF \in \bM_{\textup{Fin}}(T)$. Suppose $\cV$ is a coherent locally free sheaf  on $\bX$ and $\cW$ is a coherent locally free sheaf on $X$ of rank $r$. Then
$$ (c^* \cW \otimes_{\bX} \cV)\!-\!\rank \cF = r (\cV\!-\!\rank \cF).$$
\end{proposition}
\begin{proof}
The projection formula (Proposition~\ref{prop:projection}) ensures that 
$$ \pi_{2*} \cHom_{\bX \times T}(\pi_1^*c^*\cW \otimes_{\bX \times T} \pi_1^* \cV, \cF) \simeq 
\pi_{2*} \left(\pi_1^*\cW \otimes_{X \times T} (c\times \id)_* \cHom_{\bX \times T}(\pi_1^* \cV, \cF)\right)$$
where we have abused notation by letting $\pi_1,\pi_2$ denote projections from both $\bX \times T$ and $X \times T$. Now the support of $\cHom :=(c\times \id)_* \cHom_{\bX \times T}(\pi_1^* \cV, \cF)$ is finite, so locally on $T$ we have 
$$ \pi_1^* \cW \otimes_{X \times T} \cHom \simeq \cHom^{\oplus r}.$$
\end{proof}

Recall that the diagonal map $\Delta \colon \bX \to \bX \times \bX$ is a representable morphism which is not a monomorphism unless $\bX$ is an algebraic space. If $\bX$ is a separated DM-stack, then by definition, $\Delta$ is finite and in particular, affine. We may consider $\Delta_* \cO_{\bX}$ as a family of coherent sheaves on $\bX$ over $\bX$. To be explicit, we will consider the second factor in $\bX \times \bX$, the base space for the family. 

\begin{notation} \label{notn:bimodtensor}
Given algebraic stacks $\bX, \bY$, and quasi-coherent sheaves $\cF_X \in \Qcoh(\bX), \cF \in \Qcoh(\bX \times \bY), \cF_Y \in \Qcoh(\bY)$ we let 
$$ \cF_X \otimes_{\bX} \cF \otimes_{\bY} \cF_Y :=
\pi_X^* \cF_X \otimes_{\bX\times \bY} \cF \otimes_{\bX \times\bY} \pi_Y^*\cF_Y $$
where $\pi_X,\pi_Y$ are the projection maps. 
\end{notation}

\begin{example}  \label{eg:skyUmodG}
To get a feel for $\Delta_* \cO_{\bX}$, we consider the special case $\bX= [U/G]$ where $U$ is a quasi-projective scheme and $G$ is a finite group acting on $U$. If we wish to view coherent sheaves on $\bX$ as $G$-equivariant sheaves on $\bX$, then we should pull back $\Delta_* \cO_{\bX}$ via $\pi \times \id_{\bX} \colon U \times \bX \to \bX \times \bX$ where $\pi \colon U \to \bX$ is the canonical quotient map, and remember the $G$-action. We will study this family over $\bX$ by pulling back to a family over $U$. Hence consider the cartesian diagram
$$\begin{CD}
G \times U @>>> \bX \\
@V{\delta}VV @VV{\Delta}V \\
U \times U @>{\pi \times \pi}>> \bX \times \bX 
\end{CD}.$$
Here $\delta = (\alpha, pr_2)$ where $\alpha \colon G \times U \to U$ is the action and $pr_2$ is the projection map. Thus $\Delta_*\cO_{\bX}$ when pulled 
back to a family on $U$ is given by the $G$-equivariant sheaf $\delta_* \cO_{G \times U}$. It is useful to view this as the skew group ring $G \# \cO_U$ where left multiplication by $G$ and $\cO_U$ give the structure of a $G$-equivariant sheaf, and right multiplication by $\cO_{U}$ determines the geometry of the family.  Note that the support of $G \# \cO_U$ as a sheaf on $U \times U$ is $Z = \cup_{g \in G} (g,\id)(U)$ which is finite over $U$. Also, if we pick a geometric point $\Spec k$ of the base $U$, then the corresponding $G$-equivariant sheaf is $G \# \cO_U \otimes_U k \simeq G \# k$, the regular representation of $G$ over $k$.
In other words, $\Delta_* \cO_\bX$ is analogous to the universal family on the corresponding $G$-Hilbert scheme.
\end{example}

\begin{proposition} \label{prop:universalsky}
Suppose that $\bX$ is a separated quasi-projective stack. Then $\Delta_* \cO_{\bX}$ is a flat family of skyscraper sheaves over $\bX$. 
\end{proposition}
\begin{proof}
Now $\Delta$ is representable so the projection formula shows that for locally free $\cV_1,\cV_2 \in \Coh(\bX)$ we have
\begin{equation} \cV_1 \otimes_{\bX} \Delta_* \cO_{\bX} \otimes_{\bX} \cV_2 \simeq \Delta_*(\cV_1 \otimes_{\bX} \cV_2) .
\label{eq:projsky}
\end{equation}
Also, our assumption that $\bX$ is separated ensures now that $\Delta_*$ is exact so the same is true for $\Delta_* \cO_{\bX} \otimes_{\bX} -$ on calculating Tor. It follows that $\Delta_* \cO_{\bX}$ is indeed a flat family of coherent sheaves. 

We check now that the support $Z$ of $\Delta_* \cO_{\bX}$ is finite over $\bX$. Now $\bX$ is \'etale locally  a quotient stack of the form $[U/G]$ where $U$ is a scheme and $G$ is a finite group, so this follows from the local computation above. 

Finally, (\ref{eq:projsky}) shows that 
$$ \pi_{2*} (\cV_1 \otimes_{\bX} \Delta_* \cO_{\bX})
\simeq \pi_{2*} \Delta_*(\cV_1) = \cV_1
$$
so the discrete invariants are precisely those of a skyscraper sheaf. 
\end{proof}

\section{Tensor stable moduli stacks} \label{sc:equaliserstack}

In this section, we introduce a new moduli stack called the {\em tensor stable moduli stack}. It incorporates some of the monoidal structure of $\Coh(\bX)$ by looking at ``endomorphisms'' of the moduli stack of skyscraper sheaves induced by tensoring by line bundles. This generalises the notion of the Serre stable moduli stack introduced in \cite{CL} which corresponds to the case of the canonical line bundle. One can view this approach as a way of putting Bondal-Orlov's notion \cite{BO} of a {\em point object} on a proper moduli-theoretic footing. Indeed, in the case of a single line bundle $\cL \in \Pic(\bX)$, we will see the tensor stable moduli stack parametrises isomorophisms $\cM \simeq \cL \otimes_{\bX} \cM$ for a skyscraper sheaf $\cM$. This has two important effects. Firstly, much like the stability condition in GIT, it limits the possible sheaves $\cM$ that can occur (see \cite[Example~7.3]{CL} for an elementary example). Secondly, it changes the automorphism groups of objects (see \cite[Proposition~4.1(iii)]{CL} for an illuminating example).

Let $\bM$ be an Artin stack and $f \colon \bM \dashrightarrow \bM$ be a partially defined morphism, that is, there is an inclusion of locally closed substacks $\iota \colon \bM' \to \bM$ and a morphism $f\colon \bM' \to \bM$. Consider the graph morphism $\Gamma = \Gamma_{f} \colon \bM' \xrightarrow{\Delta} \bM' \times \bM' \xrightarrow{\iota \times f} \bM \times \bM$. We define the {\em fixed point stack of $f$} to be the fibre product stack
$$ \bM^f = \bM \times_{\Delta,\bM \times \bM, \Gamma} \bM'.$$
The definition depends of course on the domains of definition $\bM'$ which is suppressed from the notation, but like the situation with rational maps in algebraic geometry, there is often a clear ``maximal'' choice. Since stacks are themselves categories, the notion of fixed point stacks exhibits is somewhat subtle ``higher'' categorical phenomena. For example, the fixed point stack of the identity morphism is actually the inertia stack, which is not necessarily the original stack.

We will mainly be interested in the cases where $\bM$ and $\bM'$ are moduli stacks on some Grothendieck categories and the partially defined maps are induced by functors. For example, let $\bX$ be a separated quasi-projective stack, say with generator $\cG = \oplus \cG_i$ and $\bM =  \bM_{\cG-\textup{Sky}}, \tilde{\bM} = \bM_{\textup{Fin}}$. Then tensoring by a rank $r$ vector bundle $\cV$ on $\bX$, is a functor from $\Coh(\bX) \to \Coh(\bX)$ which induces a morphism of stacks $\bM \to \tilde{\bM}$. Suppose now that $\cL$ is a line bundle so tensoring by $\cL$ induces a partially defined map $\bM \dashto \bM$. We let $\bM^{\cL}$ be the fixed point stack of $\bM$ with respect to the self-map induced by tensoring by $\cL$. 

\begin{example} \label{eg:BmuRecovery}
To get a feel for the fixed point stack consider $\bX = B\mu_n := [\Spec k / \mu_n]$. The category of coherent sheaves on $\bX$ consists of finite dimensional representations of $\mu_n$. If $\chi \colon \mu_n \hookrightarrow \bG_m$ denotes the faithful character given by the inclusion homomorphism, then as per Example~\ref{eg:skyUmodG}, the only skyscraper sheaf is $\cM = \oplus_{i=0}^{n-1}\chi^i$. Using the line bundle $\chi$, we get a morphism $\bM \to \bM$ mapping $\cM$ to $\chi \otimes \cM$. Unravelling the definition of fibre products of stacks shows that the objects of $\bM^{\chi}$ are isomorphisms of $\mu_n$-modules $\theta \colon \cM \simeq \chi \otimes \cM$. Rigidification means that changing $\theta$ by a scalar does not change the object. It is not difficult to see that up to isomorphism, there is only one object in $\bM^{\chi}$ which is given by $\cM$ above and a choice of isomorphisms $\theta_i \colon \chi^i \simeq \chi \otimes \chi^{i-1}$ for $i = 1, \ldots, n-1$ and $\theta_0 \colon \chi^0 \simeq \chi \otimes \chi^{n-1}$. Let us compute the automorphism group of this object. An automorphism of $\cM = \chi^0 \oplus \chi \oplus \ldots \oplus \chi^{n-1}$ is given by $\alpha = (\alpha_0,\ldots, \alpha_{n-1}) \in \bG_m^n/\Delta$ where $\Delta$ is the diagonal copy of $\bG_m$ in $\bG_m^n$. There is an induced isomorphism $\chi \otimes \alpha \colon \chi \otimes \cM \simeq \chi \otimes \cM$, and if we use $\theta = (\theta_i)$ to  bring this back to $\cM$ we get $(\alpha_{n-1}, \alpha_0, \ldots , \alpha_{n-2})$. We are thus looking for eigenvectors $\alpha$ to  the cyclic permutation operator $\phi \colon (\alpha_0,\ldots, \alpha_{n-1}) \mapsto (\alpha_{n-1}, \alpha_0, \ldots , \alpha_{n-2})$. Now $\phi$ has $n$ linearly independent eigenvectors, all of which lie in $\bG_m^n$ so the automorphism group is the group of eigenvalues $\mu_n$. We see thus that $\bM^{\chi} \simeq \bX$ recovers the original stack. 
\end{example}

Suppose now we are given two line bundles $\cL_1,\cL_2$. Now there is a canonical natural isomorphism $\cL_1 \otimes (\cL_2 \otimes -) \simeq \cL_2 \otimes (\cL_1 \otimes -)$, so $\cL_2 \otimes (-)$ induces a partially defined map on $\bM^{\cL_1}$ and we may define iteratively $(\bM^{\cL_1})^{\cL_2}$. 

\begin{proposition}  \label{prop:doublestable}
Given a test scheme $T$, an object of $(\bM^{\cL_1})^{\cL_2}(T)$ consists of a flat family $\cM \in \bM(T)$, isomorphisms $\phi_1 \colon \cL_1 \otimes_{\bX} \cM \simeq \cM \otimes_T \cN_1, \ \ \phi_2 \colon \cL_2 \otimes_{\bX} \cM \simeq \cM \otimes_T \cN_2$ where $\cN_1, \cN_2$ are line bundles on $T$ and the $\phi_i$ are defined up to scalar only, such that the following diagram commutes up to scalar
\begin{equation} \label{eq:doublestable}
\begin{CD}
\cL_1 \otimes_{\bX} \cL_2 \otimes_{\bX} \cM @>>> \cL_2 \otimes_{\bX} \cL_1 \otimes_{\bX} \cM @>{1 \otimes \phi_1}>> \cL_2 \otimes_{\bX} \cM \otimes_T \cN_1 \\
@V{1 \otimes \phi_2}VV @. @VV{\phi_2 \otimes 1}V \\
\cL_1 \otimes_{\bX} \cM \otimes_T \cN_2 @>{\phi_1 \otimes 1}>> \cM \otimes_T \cN_1 \otimes_T \cN_2 @>>> 
\cM \otimes_T \cN_2 \otimes_T \cN_1
\end{CD}.
\end{equation}
\end{proposition}
\begin{proof}
We begin by considering an object of $\bM^{\cL_1}(T)$. By definition of the product of stacks, this consists of the data of a pair $(\cM, \cM') \in \bM(T) \times\bM(T)$ and an isomorphism $\alpha \colon (\cM, \cM) \simeq (\cM', \cL_1 \otimes \cM')$ in $\bM(T) \times \bM(T)$. Since we are working by default with rigidified moduli stacks, this reduces to the data $(\cM, \phi_1)$ where $\cM \in \bM(T)$ and $\phi_1 \colon \cL_1 \otimes_{\bX} \cM \simeq \cM \otimes_T \cN_1$ is an isomorphism for some line bundle $\cN_1$ on $T$. This isomorphism is defined only up to scalar. Now tensoring by $\cL_2$ induces a partially defined automorphism of $\bM^{\cL_1}$ which sends $(\cM,\phi_1)$ to the pair 
$$(\cL_2 \otimes_{\bX} \cM, \cL_1 \otimes_{\bX} \cL_2 \otimes_{\bX} \cM \simeq \cL_2 \otimes_{\bX} \cL_1 \otimes_{\bX} \cM \stackrel{1 \otimes \phi_1}{\simeq} \cL_2 \otimes_{\bX} \cM \otimes_T \cN_1).$$
Note that we have used the natural isomorphism $\cL_1 \otimes (\cL_2 \otimes -) \simeq \cL_2 \otimes (\cL_1 \otimes -)$ here. An object of $(\bM^{\cL_1})^{\cL_2}(T)$ consists of an isomorphism between these two pairs. This is given by $\phi_2 \colon \cL_2 \otimes_{\bX} \cM \simeq \cM \otimes_T \cN_2$ such that the diagram~(\ref{eq:doublestable}) commutes up to scalar.
\end{proof}
This description of $(\bM^{\cL_1})^{\cL_2}$ makes clear the symmetry between $\cL_1, \cL_2$ so it does not matter which order we perform the fixed point stacks. Another interesting point is that given the object of $(\bM^{\cL_1})^{\cL_2}(T)$ defined by the data $(\cM, \phi_1, \phi_2)$ above, there is a well-defined scalar $\lambda \in \cO_T^{\times}$ such that in diagram~(\ref{eq:doublestable}) we have $(\phi_2 \otimes 1) (1 \otimes \phi_1) = \lambda (\phi_1 \otimes 1) (1 \otimes \phi_2)$. Indeed, changing either $\phi_1$ or 
$\phi_2$ by a scalar does not affect $\lambda$. The formation of this scalar is compatible with pullback in the stack so there is a well-defined morphism of stacks $\nu \colon (\bM^{\cL_1})^{\cL_2} \rightarrow \bG_m$. 

\begin{definition}  \label{defn:doubletensor}
We let $\bM^{\cL_1,\cL_2}$ be the fibre product stack $(\bM^{\cL_1})^{\cL_2} \times_{\bG_m} 1$ where the map $1 \hookrightarrow \bG_m$ is the inclusion of the unit map. The data of an object over $T$ consists of $(\cM, \phi_1, \phi_2)$ as in Proposition~\ref{prop:doublestable} but now where Diagram~(\ref{eq:doublestable}) commutes on the nose. Similarly, given line bundles $\cL_1,\ldots,\cL_s$ on $\bX$, we define $\bM^{\cL_1,\ldots,\cL_s}$ and call it the {\em tensor stable moduli stack of skyscraper sheaves with respect to line bundles $\cL_1, \ldots, \cL_s$ (and generator $\cG$)}. We refer to the isomorphisms $\phi_i$ as {\em tensor stability data}.
\end{definition}

\section{Tautological moduli problem}

In this section, we show how the tensor stable moduli stack can be used to solve the tautological moduli problem in the special case where inertia groups are all abelian.

Let $\cL_1, \ldots, \cL_s$ be line bundles on a quasi-projective stack $\bX$. Let $c \colon \bX \to X$ be the canonical morphism to the coarse moduli scheme. Suppose that the geometric stabiliser groups act faithfully on $\oplus \cL_i$, in which case we say that $\oplus \cL_i$ is a faithful bundle. In this case, the geometric stabilisers are abelian, and conversely, given such a quasi-projective stack with abelian geometric stabilisers, there exists \'etale locally on $X$, a faithful direct sum of line bundles. Burnside's theorem ensures that $\bX$ has a generating sheaf $\cG$ which is  a direct sum of line bundles constructed by tensoring the $\cL_i$ together. We wish to study the tensor stable moduli stack of skyscraper sheaves $\bM^{\cL_1,\ldots,\cL_s}$ with respect to $\cL_1, \ldots, \cL_s$ and some appropriate generator $\cG$. 

The following is standard in stack theory. Recall that if $\cB$ is a $\bZ^s$-graded sheaf of algebras on $X$, then there is an action of $(\bZ^s)^{\vee} = \bG_m^s$ on $P = \underline{\Spec}_X \cB$. 
\begin{proposition}  \label{prop:XisPmodG}
With the above hypotheses, there exists a $\bZ^s$-graded sheaf of algebras $\cB = \oplus \cB_{\chi_1,\ldots,\chi_s}$ on $X$ 
such that $\bX \simeq [P/\bG_m^s]$ where $P = \underline{\Spec}_X \cB$. 
\end{proposition}
\begin{proof}
Let 
$$P = \underline{\Spec}_{\bX} \bigoplus_{(\chi_1,\ldots,\chi_s)\in \bZ^s} \cL_1^{\otimes \chi_1} \otimes_{\bX} \otimes \cdots \otimes_{\bX} \cL_s^{\otimes \chi_s}.$$
Our assumption that $\oplus \cL_i$ is faithful means that this is an algebraic space. By construction, we have $\bX \simeq [P/\bG_m^s]$. Now $P \to \bX$ is affine and hence, cohomologically affine, whilst $c \colon \bX \to X$ is cohomologically affine so the same is true of the composite $f \colon P \to X$. The algebraic space version of Serre's criterion for affineness \cite[Proposition~3.3]{Alp} ensures that $f$ is actually affine so $P = \underline{\Spec}_X \cB$ for some sheaf of algebras $\cB$. Furthermore, the action of $\bG_m^s$ on $P$ induces a $\bZ^s$-graded structure on $\cB$. 
\end{proof}
The proposition allows us to identify $\Qcoh(\bX)$ with the category $\cB-\textup{Gr}$ of $\bZ^s$-graded $\cB$-modules. Let $\chi  = (\chi_1,\ldots,\chi_s) \in \bZ^s$ which can also be viewed as a character of $\bG_m^s$. By construction, the $\bG_m^s$-equivariant sheaf $\cO_P \otimes_k \chi$ is isomorphic to $\cL_1^{\otimes \chi_1} \otimes_{\bX} \otimes \cdots \otimes_{\bX} \cL_s^{\chi_s}$. Tensoring by $\cO_P \otimes_k \chi$ corresponds to the graded shift by $\chi$ operator $M \mapsto M[\chi]$  on $\cB-\textup{Gr}$. The push forward functor $c_* \colon \cB-\textup{Gr} \to \Qcoh(X)$ from the $\cB$-module viewpoint corresponds to taking the degree 0 part. Hence,
$$ c_* \left(\cL_1^{\otimes \chi_1} \otimes_{\bX} \otimes \cdots \otimes_{\bX} \cL_s^{\otimes \chi_s}\right) \simeq \cB_{\chi} $$
is a coherent sheaf on $X$. 

\begin{theorem}  \label{thm:tautological}
Let $\bX$ be a separated quasi-projective stack and $\cL_1,\ldots, \cL_s$ be line bundles on $\bX$ such that $\cL_1 \oplus \ldots \oplus \cL_s$ is faithful. Suppose that $\cG = \oplus \cG_i$ is a generator for $\bX$ with one summand, say $\cG_0$, of the form $c^* \cV \otimes_\bX \cL$ where $\cV$ is a vector bundle on $\cV$ and $\cL$ lies in the subgroup of the Picard group generated by $\cL_1,\ldots, \cL_s$. Then $\bM^{\cL_1,\ldots\cL_s}_{\cG-\textup{Sky}} \simeq \bX$.
\end{theorem}
\begin{proof}
To simplify notation, we give the proof for the case $s=2$ and write $\bM:=\bM^{\cL_1,\cL_2}_{\cG-\textup{Sky}}$. The general case is the same and can be obtained by inserting ellipses in appropriate places. 

To construct a morphism $\bX \to \bM$, it suffices to produce a flat family of tensor stable skyscraper sheaves over $\bX$. Proposition~\ref{prop:universalsky} shows that $\Delta_* \cO_{\bX}$ is a flat family of skyscraper sheaves over $\bX$. To show this family is stable under tensoring by $\cL_i$, note that Equation~\ref{eq:projsky} gives a natural isomorphism
$$ \phi_i \colon \cL_i \otimes_{\bX} \Delta_* \cO_{\bX} \xrightarrow{\sim} \Delta_*(\cL_i) \xrightarrow{\sim} \Delta_* \cO_{\bX} \otimes_{\bX} \cL_i .$$
The data of $\Delta_* \cO_{\bX}, \phi_1, \phi_2$ thus defines a morphism from $\bX \to \bM$. 

We now construct the inverse morphism $\Phi \colon \bM \to \bX$. Recall that $\bM$ is the stackification of a pre-stack $\bM^{pre}$ whose category of sections over a test scheme $T$ can be defined as follows. An object of $\bM^{pre}(T)$ consists of a flat family of skyscraper sheaves $\cS$ with respect to $\cL_1,\cL_2$ over $T$, and isomorphisms (expressed using Notation~\ref{notn:bimodtensor}) 
$\phi_i \colon \cL_i \otimes_{\bX} \cS \simeq \cS \otimes_T \cN_i$, for line bundles $\cN_1,\cN_2$ on $T$. Our assumption on $\cG$ and Proposition~\ref{prop:Vrankcoarse} ensures that $\cS$ has skyscraper $\cL$-rank, and the isomorphisms $\phi_i$ now ensure that $\cS$ also has skyscraper $\cL'$-rank for any $\cL' \in \langle \cL_1,\ldots, \cL_s \rangle$. 

We use Proposition~\ref{prop:XisPmodG} to view $\bX$ as the quotient stack $[P/\bG_m^2]$ where $P = \underline{\Spec}_X \cB$ for some $\bZ^2$-graded sheaf of algebras $\cB$ on $X$. 
If $\cB_T$ denotes the pullback of $\cB$ to $X \times T$, then we may view $\cS$ as a $\bZ^2$-graded $\cB_T$-module and the $\phi_i$ become isomorphisms of the form 
\begin{equation}
\phi_1 \colon \cS[1,0] \simeq \cS \otimes_T \cN_1, \quad 
\phi_1 \colon \cS[0,1] \simeq \cS \otimes_T \cN_2.
\label{eq:Sshiftstable}
\end{equation}
We wish to define the object $\Phi(\cS, \phi_1,\phi_2) \in \bX(T) = [P/\bG_m^2](T)$ which will be a diagram of the form
$$
\begin{CD}
\tilde{T} @>{f}>>  P \\
@V{\pi}VV @. \\
T @. 
\end{CD}
$$
where $\pi\colon \tilde{T} \to T$ is a $\bG_m^2$-torsor and $f\colon \tilde{T} \to P$ is a $\bG_m^2$-equivariant morphism. 
We define 
$$ \cO_{\tilde{T}} = \bigoplus_{\chi_1,\chi_2 \in \bZ} \cN_1^{\otimes \chi_1} \otimes_T \cN_2^{\otimes \chi_2} \quad \text{and} \quad  \tilde{T} = \underline{\Spec}_T \cO_{\tilde{T}}$$
which is naturally a $\bG_m^2$-torsor over $T$. 

To define $f$, we will first need to define the induced map on coarse moduli schemes $\bar{f} \colon T \to X$. The isomorphisms (\ref{eq:Sshiftstable}) ensure that the sheaves $\cS_{\chi_1\chi_2} \in \Coh(X \times T)$ are all isomorphic, so in particular, have the same support $Z \subseteq X \times T$. Now $\cS$ is a flat family of skyscraper sheaves relative to $\cG$ so we know the projection map $\phi \colon Z \to T$ is a finite map and, as remarked above, $\phi_* \cS_{i_1i_2}$ are line bundles on $T$. The canonical morphism of $\cO_T$-algebras $\cO_Z \xrightarrow{\rho} \underline{\cEnd}_{\cO_T} \cS_{00} \simeq \cO_T$ splits the identity on $\cO_T$. Now $\rho$ is injective by definition of support, so $\phi \colon Z \to T$ is an isomorphism. We may now define $\bar{f}$ to be the composite $\bar{f} = \pi_1 \circ \phi^{-1}$ where $\pi_1 \colon Z \hookrightarrow X \times T \to X$ is projection onto the first factor. Note that $\cS$ is supported on the graph of $\bar{f}$ so, as a sheaf, is completely determined by $\bar{f}$ and its structure as a sheaf on $T$. 

We will define $f$ by constructing a morphism of $\bZ^2$-graded sheaves of algebras $\psi \colon \bar{f}^* \cB \to \cO_{\tilde{T}}$. Note first that the isomorphisms (\ref{eq:Sshiftstable}) show that 
$\cS \simeq \cS_{00} \otimes_T \cO_{\tilde{T}}$ as $\bZ^2$-graded sheaves on $T$. Let $\cE_{\chi}$ be the sheaf of (degree 0) graded homomorphism of sheaves $\cS \to \cS[-\chi]$ on $T$. Note that composition induces a natural algebra structure on $\cE =  \oplus_{\chi \in \bZ^2} \cE_{\chi}$. Furthermore, right multiplication on $\cS$ induces an injective  $\bZ^2$-graded homomorphism of sheaves of algebras $\cO_{\tilde{T}} \hookrightarrow \cE$. Left multiplication also induces a graded morphism of algebras $\psi \colon \cB \to \cE$, and will define our map $f$, once we show its image lies in $\cO_{\tilde{T}}$. This follows from the fact that the isomorphisms in (\ref{eq:Sshiftstable}) are isomorphisms of $\cB_T$-modules and the theory of endomorphisms compatible with shifts as explained in \cite[Section~3]{C12}. 

This completes the definition of $\Phi(\cS, \phi_1,\phi_2)$. It is now elementary, though tedious, to verify that a) this defines a morphism of pre-stacks $\bM^{pre} \to \bX$ and hence, morphism of stacks $\Phi \colon \bM \to \bX$, and that b) $\Phi$ is inverse to the morphism $\bX \to \bM$ we constructed using the universal skyscraper sheaf $\Delta_* \cO_{\bX}$. 
\end{proof}

\section{Tensor stable moduli of representations}\label{sc:tensor_stable}

In this section, we use the technology of tensor stable moduli stacks to give modular realisations of some derived equivalences. To this end, we consider a finite dimensional algebra $A$. The key case is when $A$ is the endomorphism algebra of a tilting bundle $\cT$ on a separated smooth projective stack $\bX$. The basic idea is to use tilting theory to transfer the tautological moduli problem on $\Qcoh(\bX)$ to a corresponding moduli problem on  $\textup{mod}-A$.

We assume throughout that $A$ has finite global dimension. Fix a {\em dimension vector} $\vec{d} \in K_0(A)$. If we present $A$ as a quiver with relations so that $A \simeq kQ/I$ for some quiver $Q = (Q_0,Q_1)$ and some admissible ideal $I\triangleleft kQ$, then $\vec{d}$ can be viewed as an element of $\bZ^{Q_0}$. Let $\bM_{\vec{d}}$ denote the {\em rigidified moduli stack of $A$-modules with dimension vector $\vec{d}$}. It is an Artin stack of finite type, an elementary description of which can be found in \cite[Section~2]{CL}.

Let $L_1,\ldots, L_s$ be two-sided partial tilting complexes on $A$, that is, they induce auto-equivalences 
$$- \otimes^L_A L_i \colon D^b_{fg}(A) \xrightarrow{\sim} D^b_{fg}(A)$$
where $D^b_{fg}(A)$ denotes the bounded derived category of finitely generated $A$-modules. This functor induces a partially defined map $\lambda_i \colon \bM_{\vec{d}} \dashto \bM_{\vec{d}}$ as follows. Let $\cM$ be a flat family of $A$-modules over a test scheme $T$ with dimension vector $\vec{d}$. By \cite[Proposition~3.3]{CL} (the proof given there for $L_i = DA$ works more generally in this setting), there is a locally closed subscheme $T^\circ \subseteq T$ which is the locus where a) $H_p(\cM \otimes^L_A L_i) = 0$ for $p\neq 0$ and b), $H_0(\cM \otimes^L_A L_i)$ is flat over $T^\circ$ with dimension vector $\vec{d}$. Furthermore, by \cite[Lemma~3.2]{CL}, we have
$$ H_p(\cM|_{T^\circ} \otimes^L_A L_i) \simeq H_p(\cM \otimes^L_A L_i)|_{T^\circ} .$$
As one varies $\cM$ and $T$, the locally closed subscheme $T^\circ$ determines a locally closed substack $\bM_{\vec{d}}^\circ \subseteq \bM_{\vec{d}}$. We conclude
\begin{proposition}  \label{prop:lambdamap}
The functor $- \otimes_A L_i \colon \textup{mod}-A \to \textup{mod}-A$ induces a partially defined map $\lambda_i \colon \bM_{\vec{d}} \dashto \bM_{\vec{d}}$ with domain of definition $\bM_{\vec{d}}^\circ$ in the notation above. We let $\bM_{\vec{d}}^{L_1,\ldots,L_s}$ denote the fixed point stack of these partially defined maps and call it the {\em tensor stable moduli stack} of $A$-modules with dimension vector $\vec{d}$ with respect to $L_1,\ldots, L_s$. 
\end{proposition}

The algebras of interest are those arising from tilting theory. We thus suppose that $\bX$ is a separated smooth projective stack with a tilting bundle $\cT$ which we decompose into indecomposables $\cT = \oplus_{i \in Q_0} \cT_i$. The finite dimensional algebra $A = \End_{\bX}{\cT}$ has finite global dimension and has the form $A \simeq kQ/I$ where the vertex set of the quiver $Q$ is $Q_0$. We view $\cT$ as an $(A,\cO_{\bX})$-bimodule. 
\begin{proposition}  \label{prop:Tgenerates}
The tilting bundle $\cT$ is a generating sheaf for $\bX$.
\end{proposition}
\begin{proof}
Fix a geometric point $x \in \bX$. 
The stack $[\text{pt}/\Aut(x)]$ may be presented using the Cartesian square 
\begin{equation*}
\begin{CD}
[\text{pt}/\Aut(x)] @>x>> \bX\\
@VVV @VVcV\\
\text{pt} @>x>> X.
\end{CD}
\end{equation*} 
The morphism $x$ is a closed embedding and hence $x \colon [\text{pt}/\Aut(x)] \rightarrow \bX$ is too. 
We then have that $x^* \circ x_* = \text{id}$ and so $x_*: D^b([\text{pt}/\Aut(x)]) \rightarrow D^b(\bX)$ is full and faithful. 
Let $\cV$ be a vector bundle on $[\text{pt}/\Aut(x)]$ given by some irreducible representation. 
Since $\Coh([\text{pt}/\Aut(x)])$ is semi-simple it suffices to show that $\Hom^\bullet(x^* \cT, \cV) \neq 0$. 
By adjuction we have $\Hom^\bullet(x^* \cT, \cV) = \Hom^\bullet(\cT, x_* \cV).$ 
Since $\cT$ generates the derived category we have $\Hom^\bullet(\cT, x_* \cV)= 0$ if and only if $x_* \cV = 0$. 
However this can not be true since that would imply $x^* \, x_* \cV = 0$ and hence $\cV =0$. 
\end{proof}
Let $\Phi = \RHom_{\bX}(\cT,-) \colon D^b_{c}(\bX) \to D^b_{fg}(A)$ denote the derived equivalence induced by $\cT$. 

Given line bundles $\cL_1,\ldots, \cL_s \in \Pic \bX$, we may thus consider the tensor stable moduli stack $\bM_{\cT}^{\cL_{\bullet}} := \bM_{\cT-Sky}^{\cL_1,\ldots,\cL_s}$. We now introduce the corresponding moduli stack on $\textup{mod}-A$ as follows. Firstly, we let $\vec{d}\in \bZ^{Q_0} = K_0(A)$ be defined by $d_i = \rank \cT_i$ so that for any skyscraper sheaf $\cS$ on $\bX$ relative to $\cT$, we have $\Phi(\cS)$ is an $A$-module with dimension vector $\vec{d}$. Consider the auto-equivalences
$$ \Phi \circ (\cL_i \otimes^L_{\bX} -) \circ \Phi^{-1} \colon D^b_{fg}(A) \to D^b_{fg}(A)$$
which by Rickard \cite{MR1002456} are naturally isomorphic to $- \otimes^L_A L_i$ for some two-sided partial tilting complexes $L_i$. We thus also have another tensor stable moduli stack $\bM_{\vec{d}}^{L_{\bullet}} := \bM_{\vec{d}}^{L_1,\ldots,L_s}$. 

We now examine how $\Phi = \RHom_{\bX}(\cT, -)$ induces a morphism $\phi \colon \bM_{\cT-Sky} \to \bM_{\vec{d}}$ and hence morphism $\bM_{\cT-Sky}^{\cL_{\bullet}} \to \bM_{\vec{d}}^{L_{\bullet}}$ which we also denote by $\phi$. Let $\cS$ be a flat family of skyscraper sheaves over $T = \Spec R$ relative to $\cT$ where $R$ is a noetherian ring. We use the bimodule Notation~\ref{notn:bimodtensor} for $\cS$ below. Following Grothendieck (see for example \cite[Chapter~III, Section~12]{Hart}), consider the functors 
$$ \phi^p := H^p(\cT^{\vee} \otimes_{\bX} \cS \otimes_R - ) \colon 
\textup{mod}-R \to \textup{mod}-(A \otimes R).$$
Now $\cT^{\vee}\otimes_{\bX} \cS$ has finite support over $R$ so $\phi^p = 0$ for $p >0$ and $\phi^0$ is exact. By \cite[Proposition~12.5]{Hart} and Remark~\ref{rem:Morita}, it follows that the natural transformation
$$ H^0(\cT^{\vee} \otimes_{\bX} \cS) \otimes_R (-) \to  \phi^0$$
is an isomorphism and hence $H^0(\cT^{\vee} \otimes_{\bX} \cS)$ is flat over $R$. Thus $\Phi$ is compatible with base change and we conclude 
\begin{proposition}  \label{prop:transfermap}
There is a well-defined morphism of stacks $\bM_{\cT-Sky} \to \bM_{\vec{d}}$ defined by the functor $\cS \mapsto \Phi(\cS) = H^0(\cT^{\vee} \otimes_{\bX} \cS)$ which induces the stack morphism $\phi \colon \bM_{\cT-Sky}^{\cL_{\bullet}} \to \bM_{\vec{d}}^{L_{\bullet}}$.
\end{proposition}

We wish now to show that the (quasi-)inverse functor 
$$\Psi:=(-) \otimes^L_{A} \cT \colon D^b_{fg}(A) \to D^b_{c}(\bX)$$
induces an inverse morphism $\psi \colon \bM_{\vec{d}}^{L_{\bullet}} \to \bM_{\cT-Sky}^{\cL_{\bullet}}$ to $\phi$. We need the following ampleness assumption: there is some tensor product $\cL^{\vec{i}} := \cL_1^{\otimes i_1} \otimes_{\bX} \ldots \otimes_{\bX} \cL_s^{\otimes i_s}$ such that the triple $(\Coh(\bX), \cT, (-)\otimes_{\bX} \cL^{\vec{i}})$ is ample in the sense of \cite{AZ94}. In this case we say more briefly that $(\cT, \cL_{\bullet})$ is {\em ample}.

\begin{theorem}  \label{thm:stackfromquiver}
Let $\cT = \oplus \cT_i$ be a tilting bundle on a smooth separated projective stack $\bX$ and $\cL_1, \ldots, \cL_s$ be line bundles such that $(\cT, \cL_{\bullet})$ is ample. We furthermore assume that one of the summands $\cT_i$ is isomorphic to $c^* \cV \otimes \cL$ for some vector bundle $\cV$ on $X$ and line bundle $\cL$ in the group generated by the $\cL_i$. Let $A = \End_{\bX} \cT$ and $L_1, \ldots, L_s$ be the two-sided partial tilting complexes above which correspond to $\cL_1,\ldots, \cL_s$. Then there is an isomorphism of stacks $\bX \simeq \bM^{L_{\bullet}}_{\vec{d}}$ where $\vec{d} \in K_0(A)$ is given by the rank vector of $\cT$. Furthermore, the isomorphism is given by the universal object defined by the bimodule $\ _{\cO_{\bX}}\cT^{\vee}_A$ together with some isomorphisms $ \cL_i \otimes_{\bX} \cT^{\vee} \simeq \cT^{\vee} \otimes_A^L L_i$. 
\end{theorem}
\begin{proof}
We first show that the morphism of stacks $\phi \colon \bM_{\cT-Sky}^{\cL_{\bullet}} \to \bM_{\vec{d}}^{L_{\bullet}}$ is an isomorphism in this case. Now $\phi$ is given by the equivalence $\RHom_{\bX}(\cT, -)$ so it suffices to show that the inverse equivalence $(-) \otimes^L_A \cT$ induces a well-defined morphism of stacks $\psi \colon \bM_{\vec{d}}^{L_{\bullet}} \to \bM_{\cT-Sky}^{\cL_{\bullet}}$. 

Let $T$ be a noetherian affine test scheme and consider an object of $\bM_{\vec{d}}^{L_{\bullet}}$ given by a flat family of $A$-modules $\cM$ over $T$ and isomorphisms $\theta_i \colon \cM \otimes_A^L L_i \xrightarrow{\sim} \cN_i \otimes_T \cM$ for some line bundles $\cN_i$ on $T$. Let $\cP := \cM \otimes_A^L \cT$ so for any $\vec{j} \in \bZ^s$ we have
$$\RHom_{\bX}(\cT, \cP \otimes \cL^{\otimes \vec{j}}) \simeq \cN^{\otimes \vec{j}} \otimes_T \cM \in \textup{Mod}-A.$$
Note that $\cP \in D^{\leq 0}$ is a bounded complex since $A$ has finite global dimension. Let $F_n = \Hom_{\bX}(\cT, - \otimes_{\bX} \cL^{\otimes n \vec{i}})$ where $\vec{i}$ is chosen so that $(\Coh(\bX), \cT, (-)\otimes_{\bX} \cL^{\vec{i}})$ is ample. Consider the hypercohomology spectral sequence is 
$$ R^pF_n H^q(\cP) \Rightarrow H^{p+q}(\RHom_{\bX}(\cT, \cP \otimes \cL^{\otimes n\vec{i}})).$$
For $n\gg 0$, the spectral sequence collapses to show that $F_n(H^q(\cP)) = 0$ for all $q \neq 0$. In particular, ampleness ensures that $\cP$ is concentrated in cohomological degree 0. Furthermore, the Serre module $\oplus_n F_n(\cP)$ is in sufficiently high degree equal to 
$\oplus \cN^{\otimes n \vec{i}} \otimes_T \cM$ which is flat over $T$. It follows since flatness is a property of the abelian category and \cite{AZ94}, that $\cP$ is flat over $T$ too. Furthermore, the Hilbert function of the Serre module is bounded so $\cP$ is even a flat family of finite length sheaves. The fact that it is also a family of skyscraper sheaves relative to $\cT$ follows from the fact that $\RHom_{\bX}(\cT, \cP) \simeq \cM$ which has the same rank vector as $\cT$. This completes the proof that $\psi$ is an isomorphism of stacks. Composing with the ``tautological'' isomorphism of Theorem~\ref{thm:tautological} gives the required isomorphism which maps the universal object $\Delta_{*}\cO_{\bX}$ (with appropriate isomorphisms) to $\cT^{\vee}$. 
\end{proof}

\section{Refined representation and their moduli}\label{sec:refined}

In this section, we introduce refined representations which give a different approach for recovering a DM-stack $\bX$ from a tilting bundle. The construction has antecedents in Abdelgadir-Ueda's \cite{AU} and is an outgrowth of Craw-Smith's theory of multi-linear series \cite{Craw-Smith}. 

Their starting hypotheses were quite different so, following them, we will for now relax the assumption that $\cT$ is tilting. We will however, assume that $\cT$ is a direct sum of non-isomorphic line bundles $\cT_i$. As in Section~\ref{sc:tensor_stable}, we let $A := \End(\cT)$ present $A$ as a quiver $Q = (Q_0,Q_1)$ with relations so that $A \simeq kQ/I$ for some admissible ideal $I\triangleleft kQ$. Furthermore, we view $A$-modules $\cM$ as quiver representations $(\cM_i,m_a)$ of $Q$ that satisfy the relations in $I$, the indices here range over $i \in Q_0, a \in Q_1$.
We also fix the dimension vector to be $\vec{1}:= (1,\ldots,1)$.
The letter $\bM$ will denote the rigidified moduli stack of right  $A$-modules of dimension vector $\vec{1}$.

The stack $\bM$ comes with a universal bundle $\cU  = \oplus_{i \in Q_0} \cU_i$.
Since the moduli stack $\bM$ is rigidified to remove the common $\bG_m$-stabiliser, $\cU$ is only uniquely defined up to twist by some line bundle on $\bM$.
The bundle $\cT^{\vee}$ gives a tautological family of quiver representations of $Q$ over $\bX$ and hence induces a morphism of stacks $f \colon \bX \rightarrow \bM$ such that $f^* \cU \simeq \cN \otimes_{\bX} \cT^{\vee}$ for some line bundle $\cN$ on $\bX$.
We will assume that one of the $\cT_0$ is $\cO_{\bX}$ for some distinguished vertex $0 \in Q_0$ and further arrange matters so $\cU_0 = \cN = \cO_{\bM}$, hence $\cU$ and $f^*$ are well-defined.

The universal line bundles of $\bM$ generate a free abelian subgroup of $\Pic(\bM)$ of rank $|Q_0|-1$.
This may be checked by restricting to the ``point'' of $\bM$ corresponding to the semisimple module $\cM$ of dimension vector $\vec{1}$.
Given our assumption that $\cU_0 = \cO_{\bM}$, we may identify this subgroup of $\Pic(\bM)$ with the free abelian group $\bZ^{Q_0 \setminus \{0\}}$ by taking the universal bundle $\cU_i$ to the generator $\chi_i$ for $0 \neq i \in Q_0$.
To simplify notation we write $\Lambda_Q:= \bZ^{Q_0 \setminus \{0\}}$.

To motivate the refined representation approach to stack recovery, let $\Lambda$ be a finitely generated abelian group such as $\text{im} (f^* \colon \Lambda_Q \to \Pic(\bX))$ above. Suppose its dual $\Lambda^{\vee}$ acts on an affine scheme $\Spec R$ so $R$ is a $\Lambda$-graded algebra. Consider the Artin stack $\bXtilde:= [\Spec R/\Lambda^{\vee}]$ which has {\em partial sheaf Cox ring} $\cO_{\bXtilde} \otimes_k \cO_{\Lambda^{\vee}} = \bigoplus_{\lambda \in \Lambda} \cO_{\lambda}$ where $\cO_{\lambda}$ is the line bundle on $\bXtilde$ corresponding to $R$ with $\Lambda^{\vee}$-action twisted by the character $\lambda \in \Lambda$. For any test scheme $S$, the morphisms $\phi \colon S \to \bXtilde$ consist, by definition, of a $\Lambda^{\vee}$-cover $\tilde{S} := \underline{\Spec}_S \oplus_{\lambda} \phi^* \cO_{\lambda}$ of $S$ and a $\Lambda^{\vee}$-equivariant map $\tilde{S} \to \Spec R$ which corresponds to $\Lambda$-graded algebra morphism $\phi^* \colon R \to \oplus_{\lambda} \phi^* \cO_{\lambda}$. In the multi-linear series setup, our choice of $\cT$ corresponds to picking a subset $Q_0 \subset \Lambda$ containing $0$, from which one naturally constructs the $\Lambda_Q$-graded algebra freely generated by the line bundles $\phi^*\cO_{\lambda}, \lambda \in Q_0$. To recover all the data $(\tilde{S}, \phi^*)$ however, we will in particular, need extra {\em refinement data} to reconstruct the $\Lambda$-graded sheaf of algebras $\oplus_{\lambda} \phi^* \cO_{\lambda}$. Passing from the $\Lambda_Q$-graded algebra to the $\Lambda$-graded one corresponds geometrically to lifting a map $S \to B\Lambda_Q^{\vee}$ to $B\Lambda^{\vee}$. Naturally, the kernel $\Lambda_r$ of the group homomorphism $f^* \colon \Lambda_Q \to \Pic(\bX)$ will play an important role in what follows. We now proceed with the details.

\subsection{Some monoidal notation}  \label{ssec:monoidal}
The data required to construct a partial sheaf Cox ring is most conveniently expressed using the language of monoidal categories. 

There are several monoidal categories of interest here. 
Firstly, for any stack $S$, we let $\Vect_1 (S)$ denote the symmetric monoidal category of line bundles on $S$, where the morphisms are the isomorphisms. 
The category is also rigid in the sense that it has (left and right) duals. 

Secondly, given an abelian group $\Lambda$ and subgroup $\Lambda'$, we define a symmetric monoidal category $\underline{\Lambda}/\Lambda'$ whose objects are the elements of $\Lambda$.
The morphisms are given by a pair $\lambda \in \Lambda, \lambda' \in \Lambda'$ and have the form $\lambda \xto{+\lambda'} \lambda + \lambda'$. 
Composition of morphisms is given by addition. The tensor product is also given by addition whilst the braiding and the associator are given by the identity $+0$. Note that morphisms in this category are unique (if they exist). 
This observation is useful to keep in mind when verifying diagrams in $\underline{\Lambda}/\Lambda'$ commute. 
In particular, it streamlines checking that $\underline{\Lambda}/\Lambda'$ is a symmetric monoidal category, an elementary verification we omit. Note also that it is rigid with duals given by negatives.
When $\Lambda' = 0$ we write $\underline{\Lambda} = \underline{\Lambda}/\Lambda'$.
The category $\underline{\Lambda}/\Lambda'$ is monoidally equivalent to $\underline{\Lambda/\Lambda'}$.

\begin{proposition}  \label{prop:Gcovers}
The $S$-points of the stack $B\Lambda^\vee$ for a given scheme $S$ are precisely monoidal functors $\Phi \colon \underline{\Lambda} \rightarrow \Vect_1 (S)$.
\end{proposition}
\begin{proof}
Consider first a monoidal functor $\Phi$ as above. The corresponding $S$-point of $B\Lambda^{\vee}$ will be a $\Lambda^{\vee}$-cover of the form $\underline{\Spec}_{S}\oplus_{\lambda} \Phi(\lambda)$. The data of the monoidal functor also includes the multiplication map $\Phi(\lambda) \otimes_S \Phi(\lambda') \to \Phi(\lambda + \lambda')$. Conversely, for any $S$-point of $B\Lambda^{\vee}$ corresponding to the $\Lambda^{\vee}$-cover $\pi \colon \tilde{S} \to S$, one obtains a a monoidal functor $\Phi$ where $\Phi(\lambda)$ is the $\lambda$-isotypic component of $\pi_* \cO_{\tilde{S}}$. 
\end{proof}
One important case of the construction in the above proof occurs when $\Lambda \leq \Pic(\bX)$.

\begin{definition}
Let $\Lambda \leq \Pic(\bX)$ and $\Phi \colon \underline{\Lambda} \to \Vect_1(\bX)$ be a monoidal functor such that $\Phi(\lambda)$ is a line bundle in the isomorphism class $\lambda$. The $\Lambda$-graded sheaf of $\cO_{\bX}$-algebras $\oplus \Phi(\lambda)$ above is called a {\em partial sheaf Cox ring associated to $\Lambda$}. The corresponding {\em partial Cox ring} is the induced $\Lambda$-graded ring $\oplus H^0(\bX,\Phi(\lambda))$.
\end{definition}

Given line bundles $\cM_i, i \in Q_0 \setminus \{0\}$, we have a canonical monoidal functor $\cM_? \colon\underline{\Lambda}_{Q} \rightarrow \Vect_1(S)$ taking the object corresponding to $\chi = \sum_{i \in Q_0\setminus \{0\}} a_i \chi_i$ to $\cM_{\chi}:=\otimes_{i \in Q_0 \setminus \{0\}} \cM_i^{a_i}$. To construct a $\Pic(\bX)$-graded sheaf of algebras on $S$ which is otherwise freely generated by the $\cM_i$ thus corresponds to lifting the monoidal functor $\cM_{?}$ to $\underline{\Lambda}_Q/\Lambda_r \to \Vect_1(S)$. 

\begin{proposition}  \label{prop:lifttomodLambdar}
Let $\cO_?, \cM_?|_r \colon \underline{\Lambda}_r \to \Vect_1(S)$ denote the trivial functor and the restriction of $\cM_?$ to $\underline{\Lambda}_r$. A lift of $\cM_?$ to a monoidal functor  $F \colon \underline{\Lambda}_Q/\Lambda_r \to \Vect_1(S)$ corresponds to a natural isomorphism $g \colon \cO_? \to \cM_?|_r$. The correspondence is given by 
$$ g_{\kappa}= F(0 \to \kappa) \colon \cO_S \to \cM_{\kappa}.$$
\end{proposition}
One may view the natural isomorphism $g$ above as a categorification of the relations $\kappa$ in $\Lambda_r$. 


\subsection{Definition of refined representations} \label{ssec:refined}

In this subsection, we recall the definition of refined representations introduced in \cite[Definition~3.2]{Abd}. 
However, our approach will geometric, defining the moduli stacks first. 

We start by defining a natural morphism $h\colon \bM \rightarrow B\Lambda_r^\vee$ as follows. Given a family of $A$-modules $(\cM_i,  m_a) \in \bM(S)$ with $\cM_0 = \cO_S$, we obtain a functor $\cM_? \colon \underline{\Lambda} \to \Vect_1(S)$. Restricting to $\Lambda_r \subset \Lambda_Q$ and using Proposition~\ref{prop:Gcovers} gives a morphism $S \to B\Lambda_r^{\vee}$. As $S$ varies over test schemes, we obtain a natural morphism $h \colon \bM \to B\Lambda_r^{\vee}$. 
We let $\widetilde{\bM}$ be the $\Lambda_r^{\vee}$-cover of $\bM$ defined by the following 2-Cartesian square.
\begin{equation} \label{eq:Mtilde}
\begin{CD}
\widetilde{\bM} @>>> \bM \\
@VVV @VVhV \\
\text{pt} @>>> B\Lambda_r^{\vee} = [\text{pt}/\Lambda_r^{\vee}].
\end{CD}
\end{equation}
Now the composite $S \rightarrow \text{pt} \rightarrow B\Lambda_r^\vee$
is given by the trivial functor $\cO_? \colon \underline{\Lambda}_r \rightarrow \Vect_1(S)$.
By the standard construction of the fibre product of stacks, we see that an object of $\widetilde{\bM}(S)$ corresponds to a representation $(\cM_i,m_a)$ and a natural isomorphism $g \colon \cO_? \rightarrow \cM_?|_r$ of braided monoidal functors.
In other words, $g$ is a collection of isomorphisms $g_{\kappa} \colon \cO_S \xrightarrow{\sim} \cM_{\kappa}, \kappa \in \Lambda_r$ such that $g_{\kappa + \kappa'} = g_{\kappa} \otimes g_{\kappa'}$. From Proposition~\ref{prop:lifttomodLambdar}, this allows us to lift the monoidal functor $\cM_?$ to $\underline{\Lambda}_Q/\Lambda_r$ and thus build a $\Pic(\bX)$-graded sheaf of algebras on $S$ with generators $\cM_i$. 

This can be interpreted geometrically as follows. Note that 
$$ B(\Lambda_Q/\Lambda_r)^{\vee} = \text{pt} \times_{B\Lambda_r^{\vee}} B\Lambda_Q^{\vee}.$$
From Diagram~\ref{eq:Mtilde}, we see that there is now a natural morphism $\tilde{\bM} \to B(\Lambda_Q/\Lambda_r)^{\vee}$ as desired. 

There are natural morphisms $\bX \rightarrow \text{pt}$ and $f \colon \bX \rightarrow \bM$ but to use the universal property of fibre products we further need to give 2-isomorphism between the two composites from $\bX$ to $B\Lambda_r^{\vee}$.
Given a choice of basis $B_r$ of $\Lambda_r$, this is a choice of $\gamma_{\kappa} \colon \cO \xto{\sim} \cT_{\kappa}^\vee$ for each $\kappa \in B_r$.
Such $\gamma_k$'s exist since $\Lambda_r$ is the kernel of $f^*$.
For now pick such isomorphisms; we will later further discuss this ambiguity.


\begin{definition} \label{def:refined}
A {\em flat family of refined representations} $\cM$ of $A$ over $S$ of dimension vector $\vec{1}$ consists of a representation $(\cM_i, m_a, g) \in \widetilde{\bM}(S)$ such that $g$ satisfies the following condition: for any two pairs of vertices $(i,j)$ and $(k,l)$ for which $\kappa:= (\chi_i + \chi_l)-(\chi_j + \chi_k) \in \Lambda_r$ and path $a \colon i \rightarrow j$, the following diagram commutes
\begin{equation}  \label{eq:refined}
\begin{CD}
\cM_i \otimes \cM_k @>m_a \otimes \text{id}>> \cM_j \otimes \cM_k\\
@VV\text{id}V @VVg_{\kappa}V\\
\cM_i \otimes \cM_k @>\text{id} \otimes m_{\gamma(a)}>> \cM_i \otimes \cM_l.
\end{CD}
\end{equation}
We let $\bM_{\text{ref}}$ denote the resulting {\em moduli stack of refined representations of $A$}. We refer to $g$ as {\em refinement data}.
\end{definition}

\begin{remark}
This definition of refined representations is slightly different from that given in \cite[Definition 3.2]{Abd} where the commutative diagram condition is omitted.
\end{remark}

The choice of isomorphisms $\gamma_\kappa$ for every $\kappa \in B_r$ gives a family of refined representations $(\cT_i^\vee, t_a, \gamma_\kappa)$ over $\bX$ and so a morphism $\bX \rightarrow \bM_\text{ref}$.
Any two such choices give 2-isomorphisms between the associated morphisms $\bX \rightarrow \bM_\text{ref}$.

When $\cT$ is a tilting bundle, we use the following method to select the isomorphisms $\gamma_k$ as follows. Let $K^b(\cT_i)$ is the homotopy category of complexes whose components are direct sums of line bundles $\cT_i^\vee, i \in Q_0 \setminus \{0\}$. Since $\cT$ is tilting, the inclusion functor $\iota \colon K^b(\cT_i) \rightarrow D^b(\bX)$ is a triangulated equivalence and we may pick a quasi-inverse $\iota^{-1} \colon D^b(\bX) \to K^b(\cT_i)$. 
Given $\chi \in \Lambda_{Q}$, the image of the line bundle $\cT_\chi^\vee:= f^* \cU_\chi$ under this functor $D^b(\bX) \rightarrow K^b(\cT_i)$ is by definition a complex whose components are direct sums of line bundles of the form $\cT_i^\vee$.
For $\kappa \in \Lambda_r$ the following lemma gives us an isomorphism from $\cO_\bX$ to $\cT_\kappa^\vee$.

\begin{lemma}\label{lm:kappa0}
A quasi-inverse $\iota^{-1} \colon D^b(\bX) \rightarrow K^b(\cT_i)$ of the natural inclusion $\iota\colon K^b(\cT_i) \rightarrow D^b(\bX)$ induces an isomorphism $\gamma_k\colon \cO_\bX \xto{\sim}\cT_\kappa^\vee$ for every  $\kappa \in \Lambda_r$.
\end{lemma}

\begin{proof}
The class of $\cT_\kappa^\vee$ in the Grothendieck group of $\bX$ is equal to that of $\cO_\bX$ since they are isomorphic.
Therefore the determinant of the projective resolution of $\cT_\kappa^\vee$ must be equal to $\cO_\bX$. 
Otherwise it would give a non-trivial relation in the Grothendieck group of $\bX$ which is freely generated by the classes of $\cT_i^\vee$ for $i\in Q_0$.
\end{proof}

\subsection{From tensor stable to refined representations} \label{ssc:pictoref}
For the rest of the section we will assume that $\cT$ is tilting and fix a quasi-inverse $\iota^{-1} \colon D^b(\bX) \rightarrow K^b(\cT_i)$ to $\iota$ as in Lemma~\ref{lm:kappa0}. Since $\cT$ is a direct sum of line bundles, we may consider the tensor stable moduli stack $\bM^{L_{\bullet}}$ where $L_{\bullet} = \{ L_i \,|\, i \in Q_0 \setminus \{0\}\}$ and the $L_i$ are the two-sided partial tilting complexes corresponding to $\cT_i^\vee$. We abuse notation and write $\bM^{\cT}:= \bM^{L_{\bullet}}$. 
In this subsection, we construct a morphism $h \colon \bM^{\cT} \to \bM_{\text{ref}}$ by explicitly assigning refinement data to tensor stability data.

Consider a flat family of $A$-modules $\cM \in \bM^{\cT}(S)$ over a test scheme $S$.
The tensor stability data consists of compatible isomorphisms 
$$\psi_i \colon \cM \otimes_A^L L_i \xto{\sim} \cN_i \otimes_S \cM$$
where $\cN_i$ are line bundles on $S$. The assignment $\chi_i$ to $\cN_i$ extends to a monoidal functor $\cN_? \colon \Lambda_Q \rightarrow \Vect_1(S)$ by Proposition \ref{prop:doublestable} (see Subsection~\ref{ssec:monoidal}). The isomorphisms $\psi_i$ are only defined up to scalar and we will need to show at the end, that our definition of $h(\cM,\psi_i)$ is independent of this ambiguity.

There are a several functors from $\underline{\Lambda}_{Q}$ to $\Vect_1(S)$ that are in play here: there are $\cM_?$ and $\cN_?$, which were defined above, and then $(\cM \otimes_A^L L_?)_0$ and $\cM_?^{\text{can}}$ which we will define below.
The functors $\cM_?$ and $\cN_?$ are monoidal while $(\cM \otimes_A^L L_?)_0$ and $\cM_?^{\text{can}}$ turn out not to be monoidal.
The refinement data corresponding to $\psi_i$ will be derived by relating these functors to each other.

\begin{lemma}\label{lm:equivNM}
The tensor stability data give a natural isomorphism of monoidal functors  $\xi \colon \cM_? \rightarrow \cN_?$.
\end{lemma}

\begin{proof}
First note that for $i\in Q_0$ we have that $(\cM \otimes_A^L L_i)_0 = \cM_i$ where the subscript 0 means the component corresponding to the distinguished vertex $0\in Q_0$. Furthermore $\cM_0 = \cO_S$. Our tensor stability data then gives
$$\cM_i = (\cM \otimes_A^L L_i)_0 \xto{\psi_{i,0}} \cN_i \otimes_S \cM_0 = \cN_i.$$
Tensoring over $S$ induces the required isomorphisms $\xi_\chi \colon \cM_\chi \xto{\sim} \cN_\chi$.
\end{proof}

The other two functors of interest $(\cM \otimes_A^L L_?)_0$ and $\cM_?^{\text{can}}$ are given as follows.
The functor $(\cM \otimes_A^L L_?)_0$ takes the value  $(\cM \otimes_A^L L_\chi)_0$ for $\chi \in \Lambda_Q$.
For $\cM_?^{\text{can}}$, let $K^b(P_i)$ be the homotopy category of complexes whose components are direct sums of the projectives $P_i := A\bfe_i$. Since the derived equivalence $D^b(A) \rightarrow D^b(\bX)$ takes $P_i$ to $\cT_i^\vee$, our choice of quasi-inverse $D^b(\bX) \rightarrow K^b(\cT_i)$ induces an equivalence $D^b(A) \rightarrow K^b(P_i)$. 
The image of the module $L_\chi$ under this functor $D^b(A) \rightarrow K^b(P_i)$ is by definition a complex whose components are direct sums of projectives of the form $P_i$.
We then have an expression of $(\cM \otimes_A^L L_\chi)_0$ as a complex whose components are direct sums of $\cM_i$.
Take $\cM_?^{\text{can}}$ to be the determinant of this complex of $S$-bundles.

\begin{lemma}\label{lm:Mcan}
The functors $(\cM \otimes_A^L L_?)_0$ and $\cM_?^{\textup{can}}$ are canonically isomorphic.
\end{lemma}

\begin{proof}
The complex $(\cM \otimes_A^L L_\chi)_0$ is a complex of bundles whose cohomology is a line bundle concentrated in one degree. 
There is a canonical isomorphism from this to its determinant $\cM_{\chi}^{\text{can}}$.
\end{proof}

The following example shows that, without tensor stability data, one has $\cM_\chi \neq \cM_\chi^\text{can}$ in general.

\begin{example}
Take the weighted projective line $\bX:=\bP(1,2)$. 
The stack $\bX$ is the stacky Proj of the $\bZ$-graded ring $k[y_1,y_2]$ where $\deg(y_i)=i$.
Consider the tilting bundle $\cT= \oplus_{i=0}^2 \,\cO(\vecy_i)$ and let $O(\vecy_i)$ be the two-sided partial tilting complex corresponding to $\cO(\vecy_i)$. 
Take $\cM$ here to be the universal family $\cU=(\cU_i, u_a)$ on the  moduli space of quiver representations $\bM$.
We have $(\cU \otimes_A^L O(\vecy_i))_0 = \cU_i$.
For $\chi= 2\,\chi_1$, $$\cU_\chi^\text{can} = (\cU \otimes_A^L O(\vecy_1)\otimes_A^L O(\vecy_1))_0 = (\cU \otimes_A^L O(\vecy_2))_0 = \cU_2.$$
On the other hand, $\cU_\chi = \cU_1 \otimes_\bM \cU_1$ and $\cU_2 \not\simeq \cU_1 \otimes_\bM \cU_1$ on $\bM$.
\end{example}

\begin{proposition}\label{prop:can}
Tensor stability data gives an isomorphism of functors between $\cM_?$ and $\cM_?^{\textup{can}}$ that is independent of scaling. 
\end{proposition}

\begin{proof}
For $\chi \in \Lambda_{Q}$ Lemmas~\ref{lm:equivNM} and \ref{lm:Mcan} give us
\begin{equation}\label{eq:phitog}
\cM_\chi^\text{can} \xto{\sim} (\cM \otimes_A^L L_\chi)_0 \xto{\psi_{\chi,0}} \cN_\chi \otimes_S \cM_0 = \cN_\chi \xto{{\xi_\chi^{-1}}} \cM_\chi.
\end{equation}
Scaling the stability data $(\psi_i)_{i\in Q_0}$ multiplies $\psi_\chi$ and $\xi_\chi$ by the same scalar for all $\chi \in \Lambda_Q$ and hence the result.
\end{proof}

\begin{lemma}\label{lm:kappa0_1}
Take $\kappa \in \Lambda_r$ then $\cM_0 = \cM_\kappa^\textup{can}$.
\end{lemma}

\begin{proof}
The derived equivalence $D^b(\bX) \rightarrow D^b(A)$ maps 
the isomorphisms $\gamma_{\kappa} \colon \cO_{\bX} \xto{\sim}\cT_\kappa^\vee$ in Lemma~\ref{lm:kappa0} to isomorphisms $L_0 \xto{\sim} L_{\chi}$.
\end{proof}

Putting Proposition~\ref{prop:can} and Lemma~\ref{lm:kappa0_1} together we have isomorphisms 
\begin{equation}\label{eq:gkappa}
g_\kappa \colon \cM_0 = \cM_\kappa^{\text{can}} \rightarrow \cM_\kappa
\end{equation}
for every $\kappa \in \Lambda_r$.
These give well-defined refined data on $\cM$  by Proposition~\ref{prop:doublestable}.
We therefore have a morphism $\bM^{\cT} \rightarrow \bM_\text{ref}$.

\begin{remark}
One may define the $g_\kappa$ in Equation~(\ref{eq:gkappa}) in an alternative fashion.
The isomorphism $\gamma_k \colon \cO_\bX \rightarrow \cT_\kappa^\vee$ of Lemma~\ref{lm:kappa0} induces an isomorphism $L_0 \xto{\sim} L_\kappa$ and in turn an isomorphism $(\cM \otimes_A^L L_0)_0 \xto{\sim} (\cM \otimes_A^L L_\kappa)_0$.
Noting that $\cM_0 = (\cM \otimes_A^L L_0)_0$ we may define $$g_k \colon \cM_0 = (\cM \otimes_A^L L_0)_0 \xto{\sim} (\cM \otimes_A^L L_\kappa)_0 \xto{\psi_{\kappa,0} \,\circ\, \xi_{\kappa}^{-1}} \cM_\kappa.$$
This definition coincides with the one given in Equation~\ref{eq:gkappa}.
We chose the definition in Equation~\ref{eq:gkappa} because it makes the proof of Theorem~\ref{thm:mutualinverses} more direct.
\end{remark}

\subsection{From refined representations to tensor stable} \label{ssec:reftopic}

We continue with our assumption that $\cT$ is a tilting bundle expressed as a direct sum of line bundles with $\cT_0 = \cO_{\bX}$. We have fixed $\iota^{-1} \colon D^b(\bX) \rightarrow K^b(\cT_i)$  and refinement data $\gamma_k\colon \cO_\bX \xto{\sim}\cT_\kappa^\vee$ as in Lemma~\ref{lm:kappa0}. Subsection~\ref{ssc:pictoref} yields a morphism $\bM^{\cT} \to \bM_{\textup{ref}}$. 

This morphism is rarely surjective so we first identify a locally closed subset of $\bM_\text{ref}$ that is a candidate for the image. There is a forgetful morphism $\bM_\text{ref} \rightarrow \bM$ that ignores refinement data.
We may then, as before, consider the locally closed substack $\bM^\circ$ of $\bM$ where all the $-\otimes_A L_i \colon \bM \dashto \bM$ are defined. We let $\bM_\text{ref}^\circ$ to denote the locally closed substack of $\bM_\text{ref}$ that maps to $\bM^\circ$ under the forgetful map. Note that the image of $\bM^{\cT} \rightarrow \bM_\text{ref}$ lies in $\bM_\text{ref}^\circ$.

We will work with the universal family on $\bM_\text{ref}^\circ$: this consists of universal line bundles $\cU_i$ for $i \in Q_0$, universal sections $u_a$ for $a \in Q_1$ and a lift $g \colon \underline{\Lambda}_Q/\Lambda_r \rightarrow \Vect(\bM)$ of the functor $\cU_? \colon \underline{\Lambda}_Q \rightarrow \Vect(\bM)$.
From this we aim to define tensor stability data $\psi_i \colon \cU \otimes_A^L L_i \rightarrow \cN_i \otimes_\bM \cU$ for $i\in Q_0$ and thus a morphism $\bM_\text{ref}^\circ \rightarrow \bM^{\cT}$.
We do this componentwise defining isomorphisms of sheaves $\psi_{i,j} \colon (\cU \otimes_A^L L_i)_j \rightarrow \cU_i \otimes_\bM \cU_j$.

Our triangulated equivalence to $K^b(P_i)$ gives us an expression of $(\cU \otimes_A^L L_i)_j$ as a complex whose components are direct sums of the bundles $\cU_i$.
Moreover, since our family is in $\bM_\text{ref}^\circ$ the cohomology of this complex is a line bundle concentrated in degree zero for every $i,j \in Q_0$.
Therefore $(\cU \otimes_A^L L_i)_j$ is canonically isomorphic to $\cU_{\chi}$ for some $\chi \in \Lambda_{Q}$.
The line bundle $\cU_i \otimes_\bM \cU_j$ is also of this form, it is $\cU_{\chi_i+\chi_j}$.

We compare our bundles $(\cU \otimes_A^L L_i)$ and $\cU_i \otimes_\bM \cU$ under pullback by the morphism $f \colon \bX \rightarrow \bM$: we have natural isomorphisms
\begin{equation} \label{eq:natisolinebdls}
f^*(\cU \otimes_A^L L_i) \xto{\sim} f^* \cU \otimes_A^L L_i \xto{\sim} 
\cT^{\vee} \otimes_A^L L_i \xto{\sim} \cT_i^\vee \otimes_{\bX} \cT^{\vee} \xto{\sim} 
f^*\cU_{i} \otimes_{\bX} f^* \cU.
\end{equation}
This implies that the images of $\cU_\chi \simeq (\cU \otimes_A^L L_i)_j$ and $\cU_i \otimes_\bM \cU_j$ under the composite $$\underline{\Lambda}_{Q} \xto{\cU_?} \Vect(\bM) \xto{f^*} \Vect(\bX)$$ are isomorphic.
Hence, by definition of $\Lambda_r$, the element $\chi-(\chi_i+\chi_j) \in \Lambda_r$.
The isomorphism $\psi_{i,j}$ is then defined by the composite:
\begin{equation}\label{eq:gtophi}
\psi_{i,j} \colon (\cU \otimes_A^L L_i)_j \xto{\sim} \cU_\chi \xto{g_{\chi-(\chi_i+\chi_j)}} \cU_{\chi_i + \chi_j} = \cU_i \otimes_\bM \cU_j.
\end{equation}
The isomorphisms $\psi_{i,j}$ for varying $j \in Q_0$ assemble to give the desired $A$-module isomorphism $\psi_i$. 
Indeed the commutative diagrams in Definition~\ref{def:refined} ensure compatibility with the $A$-module structure. 
This completes the construction of the morphism $\bM^\circ_{\text{ref}} \rightarrow \bM^{\Pic}$.

\begin{theorem}\label{thm:mutualinverses}
The morphisms are $\bM^\circ_{\textup{ref}} \rightarrow \bM^{\cT}$ and $\bM^{\cT} \rightarrow \bM_\textup{ref}^\circ$ are mutual quasi-inverses and give an isomorphism of stacks. 
\end{theorem}

\begin{proof}
Observe that $(\cU \otimes_A^L L_{i+j})_0 = (\cU \otimes_A^L L_i)_j$ and that $\cU_\chi$ in Equation~\ref{eq:gtophi} is $\cU^{\text{can}}_{\chi_i+\chi_j}$.
Start with refinement data $g_\kappa$ for $\kappa \in \Lambda_r$.
Then substituting Equation~\ref{eq:gtophi} in Equation~\ref{eq:phitog} gives that $\bM^\circ_{\textup{ref}} \rightarrow \bM^{\cT} \rightarrow \bM_\textup{ref}^\circ$ is the identity.
Similarly starting with tensor stability data $\psi_\chi$ for $\chi \in \Lambda_Q$, substituting Equation~\ref{eq:phitog} in Equation~\ref{eq:gtophi} gives that $\bM^\circ_{\textup{ref}} \rightarrow \bM^{\cT} \rightarrow \bM_\textup{ref}^\circ$ is the identity.
\end{proof}
Note that the construction of the inverse isomorphisms $\bM^{\cT} \simeq \bM^{\circ}_{\textup{ref}}$ is given explicitly by giving a bijective correspondence between tensor stability data and refinement data. We have thus proved Theorem~\ref{thm:samemoduli}.
\begin{corollary}\label{cor:mref}   
The stack $\bM_\textup{ref}^\circ$ is isomorphic to $\bX$.
\end{corollary}

\section{Global quotient presentation of $\bM_\textup{ref}$}\label{sc:GIT}

The moduli space of refined representations, as in the `not-refined' version, has a natural global quotient presentation.
Given a refined quiver representation $(\cM_i, m_a, g_\kappa)$ over $k$, let $\text{Func}_{\otimes, \cM_i}(\underline{\Lambda}_{Q}/\Lambda_r, \Vect(k))$ denote the set of monoidal lifts $\underline{\Lambda}_{Q}/\Lambda_r\rightarrow \Vect(k)$ of the functor $\cM_? \colon \underline{\Lambda}_{Q} \rightarrow \Vect(k)$.
Picking a basis for $\cM_i$ to identify them with $k$ enables us to view the $m_a$ as elements of $\bA^1$ and $g$ as elements of $\Lambda_r^{\vee}$. We thus obtain an element of
\begin{equation*}
\cR(Q) := \oplus_{a \in Q_1} \Hom(\cM_{t(a)}, \cM_{h(a)}) \times \text{Func}_{\otimes, \cM_i}(\underline{\Lambda}_{Q}/\Lambda_r, \Vect(k))\simeq \bA^{Q_1} \times \Lambda_r^{\vee}. \end{equation*}
We use $\cR_A$ to denote the {\em refined representation space} which is the closed subscheme of $\cR(Q)$ whose elements descend to representations of $A \simeq kQ/I$ and furthermore make the diagrams in (\ref{eq:refined}) commute.
The gauge group
\begin{equation*}
\GL(\vec{1}) := \oplus_{i \in Q_0} \GL(\cM_i) \simeq \bG_m^{Q_0}.
\end{equation*}
naturally acts on $\cR_A$ by change of basis.
Note that the diagonal one-parameter subgroup 
$$\Delta = \{(\lambda, \ldots,\lambda) | \lambda \in \bG_m\} \leq \GL(\vec{1})$$ 
acts trivially so there is an induced action of $\PGL(\vec{1}) := \GL(\vec{1}) / \Delta$.
Taking the quotient by $\PGL(\vec{1})$ as opposed to $\GL(\vec{1})$ amounts to considering the rigidified moduli stack as opposed to the unrigidified version. 
An argument similar to the proof of \cite[Proposition 3.9]{Abd} shows that $\bM_{\text{ref}}$ is isomorphic to the quotient stack $[\cR_A/ \PGL(\vec{1})]$.

\subsection{$\bM_{\text{ref}}^\circ$ and GIT stability}

For this subsection we assume $\cT$ is tilting.
We study the locus $\bM_{\text{ref}}^\circ$, i.e.\ the one isomorphic to $\bX$, using GIT. First we make the following observation.

\begin{lemma}\label{lm:open}
The subset $\bM_{\textup{ref}}^\circ \subset \bM_{\textup{ref}}$ is open.
\end{lemma}

\begin{proof}
Fix a pair of vertices $i,j$ and consider the object $(\cU \otimes^L_A L_j)_i \in D^b(\bM_\text{ref})$ where $\cU$ is the universal family on  $\bM_{\text{ref}}$. This is a complex is of rank one.
First note that when restricted to the image of $\bX$, this complex is concentrated in degree zero and is quasi-isomorphic to the line bundle $\cT_i^{\vee} \otimes_{\bX} \cT_j^\vee$. This implies that the cohomology sheaves at nonzero degrees are torsion and furthermore, that the cohomology sheaf in degree zero is of rank one.
The nonzero cohomology sheaves and the torsion part of the zero cohomology sheaf are supported on a closed set that does not contain the image of $X$. The complement of this subset is $\bM_{\text{ref}}^\circ$ and hence the result.
\end{proof}

Lemma~\ref{lm:open} hints that perhaps $\bM^\circ_{\textup{ref}}$ is carved out by some GIT stability parameter. 

As discussed above, $\bM_{\text{ref}}$ is isomorphic to the quotient stack $[\cR_A/ \PGL(\vec{1})]$.
Following King \cite{Ki}, one may define an intrinsic notion of $\theta$-stability for refined representations equivalent to the GIT stability of the $\PGL(\vec{1})$ action on $\cR_A$,  see \cite[Definition 3.4]{Abd}. 
Observe that $\PGL(\vec{1})$ maybe identified with the subgroup $\prod_{i\neq 0} \GL(\cM_i) \subset \GL(\vec{1})$.
This in turn identifies $\Lambda_Q$ with the characters of $\PGL(\vec{1})$ so we may consider the $\theta$-semistable points in $\cR_A$ for $\theta \in \Lambda_Q$.
We will use $\bM_{\textup{ref}}^\theta$ to denote the $\theta$-semistable locus of $\bM_{\textup{ref}}$.

\begin{definition}
Given a generic stability parameter $\theta \in \Lambda_Q$ we say {\em $\theta$ stabilises $\bX$} if $f(x)$ is $\theta$-stable for all $x \in \bX$.
In other words, {\em $\theta$ stabilises $\bX$} if $f \colon \bX \rightarrow \bM_{\textup{ref}}$ factors through $\bM_{\textup{ref}}^\theta$.
\end{definition}

\begin{theorem}\label{thm:stab}
Let $\theta \in \Lambda_Q$ be a stability parameter that stabilises $\bX$ then the stacks $\bM_{\textup{ref}}^\circ \simeq \bX$ are isomorphic to a connected component of $\bM_{\textup{ref}}^\theta$.
\end{theorem}

\begin{proof}
The fact that $\theta$ stabilises $\bX$ gives a morphism $f \colon \bX \rightarrow \bM_{\text{ref}}^\theta$.
This combined with Corollary~\ref{cor:mref} and Lemma~\ref{lm:open} gives that $f$ is an open embedding of $\bX$ in $\bM_{\text{ref}}^\theta$.

We also have that the image of the coarse moduli space $X$ under the coarse moduli map induced by $f$ is closed in the coarse moduli space of $\bM_\text{ref}^\theta$. 
Since $\bX$ is embedded in $\bM_\text{ref}^\theta$ this implies that $f$ is also a closed embedding.
The result follows.
\end{proof}

\begin{remark}
The existence of a $\theta$ that stabilises $\bX$ is part of the hypothesis of Theorem~\ref{thm:stab}.
We expect such $\theta \in \Lambda_Q$ to exist in general.
In fact, in Lemma~\ref{lm:generic} below, we earmark a candidate.
\end{remark}

\begin{lemma}\label{lm:generic}
There exists a stability condition $\theta \in \Lambda_Q$ so that $\theta$ is generic, i.e.\ $\theta$-semistable implies $\theta$-stable, and $f^*(\theta)$ is the pullback of a very ample line bundle on the coarse moduli space $X$ of $\bX$.
\end{lemma}

\begin{proof}
Definition 3.4 of \cite{Abd} only tests $\theta \in \Lambda_{Q}$ against filtrations $\cM_\bullet$ of $(\cM_i, m_a)$ that satisfy $\kappa(\cM_\bullet)=0$ for all $\kappa\in \Lambda_r \subset \Lambda_{Q}$.
In other words, stability is only dependent on the class of $\theta$ in $\Lambda_{Q}/\Lambda_r$.
Now $\cT$ generates $D^b(\bX)$ so the homomorphism $f^*\colon \Lambda_{Q} \rightarrow \Pic(\bX)$ is surjective and we have an isomorphism $\Lambda_{Q}/\Lambda_r \simeq \Pic(\bX)$.
Furthermore, ampleness is a generic condition in $\Pic(X)_\bQ \simeq \Pic(\bX)_\bQ$.
Therefore, if necessary, one may perturb $\theta$ so that it is generic and $f^*(\theta)$ is pulled back from the ample cone of $X$.
\end{proof}

\begin{remark}\label{rm:generic}
The statement corresponding to Lemma~\ref{lm:generic} is not true when applied to the moduli of quiver representations of the tilting quiver of a general projective DM stack.
\end{remark}

\subsection{Mori-dream stacks}\label{ssec:nice}
In \cite{AU}, Abdelgadir-Ueda recover weighted projective lines as moduli stacks of refined representations. The method of proof there is quite different from the approach in Subsections~\ref{ssc:pictoref}, \ref{ssec:reftopic} and is based on recovering the Cox ring of $\bX$ (see the introductory blurb to Section~\ref{sec:refined}). In this subsection, we give a general account of this method of proof. What follows does not require the tilting assumption on $\cT$ so we drop it from here on. We will though, assume that $\bX$ is a Mori-dream stack, i.e.\ that $\bX$ has a finitely generated Cox ring $R$ and that $$\bX \simeq \bigg[\frac{\Spec(R) \setminus V(B_\theta)}{\Pic(\bX)^\vee}\bigg]$$ for $B_\theta$ the irrelevant ideal given by some ample line bundle $\theta \in \Pic(\bX)$. Note that the Peirce components $\bfe_i A \bfe_j$ of $A$ are all given by the corresponding graded component of $R$. 

The moduli space $\bM$ is also a global quotient (see \cite[Sections 4 \& 5]{Ki}), we spell this out here to set notation.
Fixing $k$ as the vector spaces on the vertices, every point of $\bA^{Q_1}$ defines a representation of $Q$.
Those that descend to give $A$-modules form a closed subset that we will denote $\Spec(S) \subset \bA^{Q_1}$.
The ring $S$ is naturally graded by $\Lambda_Q$, in fact the graded component $S_{j-i}$ has a natural $k$-basis given by $\bfe_i A \bfe_j$. 
We then have that $\bM \simeq [\Spec(S)/\PGL(\vec{1})]$.
The morphism $f \colon \bX \rightarrow \bM$ is induced from the  tautological homomorphism of partial Cox rings $h \colon S \rightarrow R$ with the group homomorphism $f^* \colon \Lambda_Q \rightarrow \Pic(\bX)$ intertwining the grading.

From the introduction to Section~\ref{sec:refined}, there are two issues to address for stack recovery in the case of Mori-dream stacks. Firstly, given a family $(\cM_i, m_a)$ of $A$-modules, the issue of reconstructing the $\Pic(\bX)$-graded algebra $\oplus_{\chi \in \Pic(\bX)} \cM_{\chi}$ is resolved by introducing refinement data. We now address the other issue of matching up the algebra morphism $R \to \oplus_{\chi \in \Pic(\bX)} \cM_{\chi}$ with the data of the $A$-module action on $\oplus_i \cM_i$. As discussed just above Definition ~\ref{def:refined}, for any two pairs of elements $(i,j)$ and $(k,l)$ of $Q_0$ for which $(\chi_i + \chi_l)-(\chi_j + \chi_k) \in \Lambda_r$ we have an isomorphism of Peirce components $\gamma(i,j;k,l) \colon \bfe_i A \bfe_j \xrightarrow{\sim} \bfe_k A \bfe_l$. 
This induces an isomorphism $\gamma(i,j;k,l) \colon S_{j-i} \xrightarrow{\sim} S_{l-k}$.
Hence, graded components of the Cox ring $R$ may become ``separated'' in the algebra $A$ and hence also in $S$. This separation of components gives rise to some obvious elements of $\ker h$ that have the form 
\begin{equation}  \label{eq:deindex}
    y - \gamma(i,j;k,l) (y)
\end{equation}
where $y \in S_{j-i}$ and $i,j,k,l \in Q_0$ are such that $(\chi_i + \chi_l)-(\chi_j + \chi_k) \in \Lambda_r$. We let $I_{\textup{de}}\subseteq \ker h$ be the $\Pic(\bX)$-graded ideal generated by these elements.

\begin{definition} \label{def:captureCox}
We say our bundle $\cT$ {\em captures the Cox ring of $\bX$} if $h$ is surjective and descends to an isomorphism $S/I_{de} \simeq R$.
\end{definition}

In the following we demonstrate how the isomorphism $\bX \simeq \bM_{\text{ref}}^\theta$ can be obtained from the ring homomorphism $h$ when $\cT$ captures the Cox ring of $\bX$.
To further clarify, this is precisely when the difference between $S$ and $R$ is only due to the separation of Peirce components discussed above.

Note that $\PGL(\vec{1}) = \Lambda_Q^\vee$, so the $\Lambda_r^{\vee}$-cover of $\bM$ introduced in Subsection~\ref{ssec:refined} is $\tilde{\bM} = [\Spec S/ \Pic(\bX)^{\vee}]$. This can be expressed as a $\PGL(\vec{1})$-quotient as follows. Now $\PGL(\vec{1})$ acts diagonally on $\Spec S \times \Lambda_r^\vee$ so $\tilde{\bM} = [\Spec (S) \times \Lambda_r^{\vee} / \PGL(\vec{1})]$. Furthermore, 
$$[\Spec R/\Pic(\bX)^\vee] \simeq [V/\PGL(\vec{1})]$$ where $V$ is the $\PGL(\vec{1})$-orbit of $\Spec R \times 1 \subseteq \Spec S \times \Lambda_r^\vee$.
The refined representation space $\cR_A$ is a closed subscheme of $\Spec S \times \Lambda_r^\vee$ cut-out by the commutative diagrams (\ref{eq:refined}).
We'll eventually show that in this setting $\cR_A = V$.

We clarify what $V$ is. 
Let $\frakm\triangleleft k\Lambda_r$ be the maximal ideal corresponding to $1 \in \Lambda_r^\vee$. 
Then $V$ is the closed subscheme defined by the ideal sheaf
$$ I_V := \bigcap_{\tau \in \PGL(\vec{1})} \tau\cdot\left(I \otimes k\Lambda_r + S \otimes \frakm \right)$$
To compute this, we need some notation. 
Recall that $I$ is only $\Pic(\bX)$-homogeneous. Let $s \in I$ be a $\Pic(\bX)$-homogeneous element. 
Then we can express it as a sum of $\Lambda_Q$-homogeneous elements $s = s_1 + \ldots + s_m$ all of whose degrees lie in the same coset of $\Lambda_r$. 
Hence we can pick $\kappa_i \in \Lambda_r$ such that $s^h := s_1 \otimes \kappa_1 + \ldots + s_m \otimes \kappa_m$ is $\Lambda_Q$-homogeneous. 
The {\em homogenisation} $s^h$ is only defined up to multiplication by some $\kappa \in \Lambda_r$. 

\begin{lemma}
Let $\Sigma \subset I$ be a set of $\Pic(\bX)$-homogeneous generators of $I$ then $I_V = \langle s^h \,|\, s \in \Sigma \rangle.$ 
\end{lemma}

Under the assumption that $\cT$ captures the Cox ring of $\bX$ we have that $I$ is generated by elements of the form $s:= y - \gamma (y)$ for some $\Lambda_Q$-homogeneous element $y \in S$.
Its homogenisation $s^h$ can be encoded moduli-theoretically on the refined representation space as follows.
Take $(\cM_i, m_a, g)$ a $k$-point of $\bM_\text{ref}$.
Suppose that $i,j,k,l \in Q_0$ are such that $\kappa:= (\chi_i + \chi_l)-(\chi_j + \chi_k) \in \Lambda_r$. 
The refinement data $g$ thus gives an isomorphism $g_{\kappa} \colon \Hom_k(\cM_i,\cM_j) \to \Hom_k(\cM_k,\cM_l)$. 
Suppose $a \in \bfe_i A \bfe_j$ so we have a  multiplication by $a$ map $m_a \colon \cM_i \to \cM_j$ and multiplication by $\gamma(a)$ map $m_{\gamma(a)} \colon  \cM_{k} \to \cM_{l}$. Then the relation $s^h$ corresponds to commutativity of the following diagram 
\begin{equation}
\begin{CD}
k @>m_a \otimes \cM_{i}^\vee>> \cM_{j} \otimes \cM_{i}^\vee\\
@VV\text{id}V @VV{g_\kappa}V\\
k @>m_{\gamma(a)}\otimes \cM_{k}^\vee>> \cM_{l} \otimes \cM_{k}^\vee.
\end{CD}
\end{equation}
This is precisely the tensor product of the diagram (\ref{eq:refined}) by $(\cM_i \otimes \cM_k)^\vee$.
Hence $\cR_A = V$.

\begin{corollary}
If the bundle $\cT$ captures the Cox ring of $\bX$ then the ring homomorphisms $h$ along with the intertwining homomorphism $f^*$ give an isomorphism of stacks $$\bigg[\frac{\Spec R}{\Pic(\bX)^\vee}\bigg] \simeq \bigg[\frac{\cR_A}{\PGL(\vec{1})}\bigg].$$
\end{corollary}

\begin{corollary}  \label{cor:stackfromCox}
There is a GIT parameter $\theta \in \Lambda_Q$ for which $\bX \simeq \bM_\textup{ref}^\theta$.
\end{corollary}

\begin{example}
Insisting that $\cT$ is tilting does not guarantee that it captures the corresponding Cox ring: take $\bX= \bF_2:= \bP(\cO_{\bP^1} \oplus \cO_{\bP^1}(2))$ with the well-known tilting bundle $\cT= \cO \oplus \cO(f) \oplus \cO(h) \oplus \cO(f+h)$ (see for example \cite[Proposition~6.3]{C17}) where $f$ and $h-2f$ are the divisors corresponding to the fibre and the $(-2)$-section. 
\end{example}

\section{An example: GL projective spaces}\label{sec:HIMO}

Here we consider Geigle-Lenzing (GL) projective space as defined by Herschend-Iyama-Minamoto-Oppermann in \cite{HIMO}.
First we begin by briefly recalling their construction. 

The building blocks for GL projective spaces are: a polynomial ring $C:= k[t_0, \ldots, t_d]$ and $n$ linear forms $l_0,\ldots,l_n \in C$ each with an assigned integer weight $p_i \geq 2$. 
The forms are required to be in general position, i.e.\ every subset of at most $d+1$ forms is linearly independent. 
Define $$R':= k[t_0,\ldots, t_d, y_0, \ldots, y_n],  \quad I := \langle y_i^{p_i} - l_i \, | \, 0\leq i \leq n \rangle \subset R'$$ and take $R:= R'/I$. 
The ring $R$ is graded by the abelian group $$\bL:= \bZ \vecy_0 \oplus \cdots \oplus \bZ \vecy_n \oplus \bZ \vecc \,/\, \langle p_i \vecy_i - \vecc \, | \, 0\leq i \leq n \rangle$$ with $\deg(t_i)=\vecc$ and $\deg(y_i) = \vecy_i$. 
The corresponding stack is then given by $$\bX:= \bigg[\frac{\Spec(R) \setminus \{0\}}{\bL^\vee}\bigg].$$ 
The Picard group of $\bX$ is given by $\bL$ and has a partial ordering defined by $$\vecy \leq \vecz \iff \Hom(\vecy, \vecz) \neq 0.$$

The ring $R$ may be expressed in a slightly different fashion. 
We may assume that $n\geq d$ by adjoining an indeterminant $y_i$ for every extra $t_i$ variable and setting $l_i= t_i$ and $p_i=1$. 
Now the genericity assumption allows us to change variables so that the first $d+1$ forms $l_i$ are just $t_i$.
The ring $R$ is then naturally isomorphic to a quotient of $k[y_0, \ldots, y_n]$ by $n-d$ linear relations in the monomials $y_i^{p_i}$.
Keeping this presentation of $R$ in mind may make reading the rest of the section easier.

One may also think of $\bX$ as an iterated root stack, as observed in \cite[Observation 3.1.4]{HIMO}. This is done as follows: $\bX_0:= \bP^d$, $\bX_i := \bX_{i-1 (\cO(\vecc),l_i,p_i)}$ with $\bX_n \cong \bX$.

\begin{theorem}\cite[Theorem 6.1.2]{HIMO}
The following object is tilting in $D^b(\bX)$: $$\cT:= \bigoplus_{\vecy \in [0,d\vecc]} \cO(\vecy).$$
\end{theorem}

Let $Q$ be the quiver of sections of the collection of line bundles $$\{\cO(\vecy)\,|\,\vecy \in [0,d\vecc]\}\subset \Pic(\bX).$$
The algebra $\End(\cT)$ is then a quotient of $kQ$ by some ideal $I$.
We will use $\vecy$ to denote the vertex of $Q$ corresponding to the line bundle $\cO(\vecy)$ and $a_{\vecm+i}$ to denote the arrow corresponding to the section $y_i$ from $\vecm$ to $\vecm+\vecy_i$ for $\vecm \in [0, d\vecc - \vecy_i]$.
Take $\bM$ to be the moduli space of quiver representations of $kQ/I$ with dimension vector $\vec{1}$ and let $S$ be its Cox ring.

As in Subsection~\ref{ssec:nice}, we have a tautological homomorphism $S \rightarrow R$.
Notice that every arrow in $Q$ is of this form $a_{\vecm+i}$ for $\vecm \in [0, d\vecc - \vecy_i]$.
Therefore for every $\vecm \in [0, d\vecc - \vecy_i]$ and $0\leq i\leq n$ there we have a generator $y_{a_{0+i}} - y_{a_{\vecm+i}}$ in the corresponding ideal $I_{de}$.
Thus $S/I_{de}$ is a quotient of $k[y_0,\ldots, y_n]$: the kernel of $S/I_{de} \rightarrow k[y_{a_0},\ldots, y_{a_n}]$ is generated by the $n-d$ relations between the arrows with head at $0$ and tail at $\vec{c}$ but those are the equations between the monomials $y_{i}^{p_i}$ in $R$.
Hence $S/I_{de} \simeq R$ and $\cT$ captures the Cox ring of $\bX$.
We thus have:

\begin{corollary} \label{cr:HIMO}
For the GL projective space $\bX$, we have $\bX \simeq \bM^{\cT} \simeq \bM^\theta_\textup{ref}$ for some $\textup{GIT}$ stability parameter $\theta$.
\end{corollary}

\bibliographystyle{amsplain}
\bibliography{references}

\providecommand{\bysame}{\leavevmode\hbox to3em{\hrulefill}\thinspace}
\providecommand{\MR}{\relax\ifhmode\unskip\space\fi MR }
\providecommand{\MRhref}[2]{%
  \href{http://www.ams.org/mathscinet-getitem?mr=#1}{#2}
}
\providecommand{\href}[2]{#2}
\begin{thebibliography}{10}

\bibitem{AU}
Tarig Abdelgadir and Kazushi Ueda, \emph{Weighted projective lines as fine
  moduli spaces of quiver representations}, Comm. Algebra \textbf{43} (2015),
  no.~2, 636--649. \MR{3274027}

\bibitem{Abd}
Tarig M.~H. Abdelgadir, \emph{Quivers of sections on toric orbifolds}, J.
  Algebra \textbf{349} (2012), 128--149. \MR{2853630}

\bibitem{ACV}
Dan Abramovich, Alessio Corti, and Angelo Vistoli, \emph{Twisted bundles and
  admissible covers}, Comm. Algebra \textbf{31} (2003), no.~8, 3547--3618,
  Special issue in honor of Steven L. Kleiman. \MR{2007376}

\bibitem{AOV}
Dan Abramovich, Martin Olsson, and Angelo Vistoli, \emph{Tame stacks in
  positive characteristic}, Ann. Inst. Fourier (Grenoble) \textbf{58} (2008),
  no.~4, 1057--1091. \MR{2427954}

\bibitem{Alp}
Jarod Alper, \emph{Good moduli spaces for {A}rtin stacks}, Ann. Inst. Fourier
  (Grenoble) \textbf{63} (2013), no.~6, 2349--2402. \MR{3237451}

\bibitem{ATV90}
M.~Artin, J.~Tate, and M.~Van~den Bergh, \emph{Some algebras associated to
  automorphisms of elliptic curves}, The {G}rothendieck {F}estschrift, {V}ol.
  {I}, Progr. Math., vol.~86, Birkh\"{a}user Boston, Boston, MA, 1990,
  pp.~33--85. \MR{1086882}

\bibitem{AZ94}
M.~Artin and J.~J. Zhang, \emph{Noncommutative projective schemes}, Adv. Math.
  \textbf{109} (1994), no.~2, 228--287. \MR{1304753}

\bibitem{AZ01}
\bysame, \emph{Abstract {H}ilbert schemes}, Algebr. Represent. Theory
  \textbf{4} (2001), no.~4, 305--394. \MR{1863391}

\bibitem{MR2421120}
Aaron Bergman and Nicholas~J. Proudfoot, \emph{Moduli spaces for {B}ondal
  quivers}, Pacific J. Math. \textbf{237} (2008), no.~2, 201--221. \MR{2421120}

\bibitem{BO}
Alexei Bondal and Dmitri Orlov, \emph{Reconstruction of a variety from the
  derived category and groups of autoequivalences}, Compositio Math.
  \textbf{125} (2001), no.~3, 327--344. \MR{1818984}

\bibitem{C12}
Daniel Chan, \emph{Twisted rings and moduli stacks of ``fat'' point modules in
  non-commutative projective geometry}, Adv. Math. \textbf{229} (2012), no.~4,
  2184--2209. \MR{2880219}

\bibitem{C17}
\bysame, \emph{2-hereditary algebras and almost {F}ano weighted surfaces}, J.
  Algebra \textbf{478} (2017), 92--132. \MR{3621664}

\bibitem{CL}
Daniel Chan and Boris Lerner, \emph{Moduli stacks of {S}erre stable
  representations in tilting theory}, Adv. Math. \textbf{312} (2017), 588--635.
  \MR{3635820}

\bibitem{Craw-Smith}
Alastair Craw and Gregory~G. Smith, \emph{Projective toric varieties as fine
  moduli spaces of quiver representations}, Amer. J. Math. \textbf{130} (2008),
  no.~6, 1509--1534. \MR{2464026 (2010b:14101)}

\bibitem{GL}
Werner Geigle and Helmut Lenzing, \emph{A class of weighted projective curves
  arising in representation theory of finite-dimensional algebras},
  Singularities, representation of algebras, and vector bundles ({L}ambrecht,
  1985), Lecture Notes in Math., vol. 1273, Springer, Berlin, 1987,
  pp.~265--297. \MR{915180}

\bibitem{Hart}
Robin Hartshorne, \emph{Algebraic geometry}, Springer-Verlag, New
  York-Heidelberg, 1977, Graduate Texts in Mathematics, No. 52. \MR{0463157}

\bibitem{HIMO}
Martin Herschend, Osamu Iyama, Hiroyuki Minamoto, and Steffen Oppermann,
  \emph{Representation theory of {G}eigle-{L}enzing complete intersections},
  arXiv:1409.0668, 2014.

\bibitem{KeM}
Se\'an Keel and Shigefumi Mori, \emph{Quotients by groupoids}, Ann. of Math.
  (2) \textbf{145} (1997), no.~1, 193--213. \MR{1432041}

\bibitem{Ki}
A.~D. King, \emph{Moduli of representations of finite-dimensional algebras},
  Quart. J. Math. Oxford Ser. (2) \textbf{45} (1994), no.~180, 515--530.
  \MR{1315461}

\bibitem{Kr}
Andrew Kresch, \emph{On the geometry of {D}eligne-{M}umford stacks}, Algebraic
  geometry---{S}eattle 2005. {P}art 1, Proc. Sympos. Pure Math., vol.~80, Amer.
  Math. Soc., Providence, RI, 2009, pp.~259--271. \MR{2483938}

\bibitem{Lurie}
Jacob Lurie, \emph{Tannaka duality for geometric stacks}, arXiv:math/0412266,
  2004.

\bibitem{OS}
Martin Olsson and Jason Starr, \emph{Quot functors for {D}eligne-{M}umford
  stacks}, Comm. Algebra \textbf{31} (2003), no.~8, 4069--4096, Special issue
  in honor of Steven L. Kleiman. \MR{2007396}

\bibitem{MR1002456}
Jeremy Rickard, \emph{Morita theory for derived categories}, J. London Math.
  Soc. (2) \textbf{39} (1989), no.~3, 436--456. \MR{1002456}

\end{thebibliography}

\end{document}